\documentclass[a4paper,12pt]{amsart}

\usepackage[utf8x,]{inputenc} 
\usepackage{mathtools}
\usepackage{amsthm, amsmath, latexsym, amsfonts, tikz}
\usepackage[all]{xy}

\usepackage{amssymb}
\usepackage{mathrsfs}
\usepackage{enumitem}
\usepackage{amscd}
\usepackage{tikz-cd}
\usepackage{booktabs}
\usepackage{tikz-qtree}

\numberwithin{equation}{section}
\theoremstyle{plain}
\newtheorem{teo}[equation]{Theorem}
\newtheorem{lem}[equation]{Lemma}
\newtheorem{prop}[equation]{Proposition}
\newtheorem{cor}[equation]{Corollary}
\newtheorem{conj}{Conjecture}
\newtheorem*{notation}{Notation and conventions}

\theoremstyle{remark}
\newtheorem{rem}[equation]{Remark}
\theoremstyle{plain}
\newtheorem{defn}[equation]{Definition}

\newtheorem*{reftheorem*}{Theorem \reftotheorem}
\newenvironment{reftheorem}[1]
  {\newcommand{\reftotheorem}{\ref{#1}}\begin{reftheorem*}}
  {\end{reftheorem*}}

\newtheorem{nonumberingt}{Acknowledgements}

\DeclareMathOperator{\Grass}{Gr}
\DeclareMathOperator{\Pic}{Pic}
\DeclareMathOperator{\Pj}{\mathbb{P}}
\DeclareMathOperator{\id}{id}
\DeclareMathOperator{\rank}{rank}
\DeclareMathOperator{\Gr}{Gr}
\DeclareMathOperator{\Aut}{Aut}
\DeclareMathOperator{\cl}{cl}
\DeclareMathOperator{\alg}{alg}
\DeclareMathOperator{\tr}{tr}
\DeclareMathOperator{\Fix}{Fix}

\title{Algebraic cycles on Todorov surfaces of type $(2,12)$}
\author[Natascia Zangani]{Natascia Zangani}
\email{natasciazangani@gmail.com}
\begin{document}
\maketitle

\begin{abstract}
We focus on Voisin's conjecture on 0--cycles on the self--product of surfaces of geometric genus one, which arises in the context of the Bloch--Beilinson filtration conjecture. We verify this conjecture for the family of Todorov surfaces of type $(2,12)$, giving an explicit description of this family as quotient surfaces of the complete intersection of four quadrics in $\mathbb{P}^{6}$. We give some motivic applications.
\end{abstract}
\section{Introduction}
The study of the influence  of Chow groups on singular cohomology is motivated by Mumford's theorem on 0--cycles and has been investigated extensively, whereas the converse influence is rather conjectural. For example, Bloch's conjecture is still open and the Bloch--Beilinson's filtration conjecture is still far from being solved. In the spirit of exploring this influence, Voisin formulated in 1996 the following conjecture on 0--cycles on the self--product of surfaces of geometric genus one, which is implied by the generalized Bloch conjecture adapted to motives (see \cite[Conjecture 4.33]{voisin2014chow} and \cite[Section 4.3.5.2]{voisin2014chow} for an analysis of this relation).

\begin{conj}[Voisin \cite{voisin1996remarks}]
\label{conj:Voi}
Let $S$ be a smooth complex projective surface with $h^{2,0}(S)=p_{g}(S)=1$ and $h^{1,0}(S)=q(S)=0$. Let $a,a'\in A^{2}_{hom}(S)$ be two $0$--cycles of degree 0 (i.e. homologically trivial $0$--cycles). Let $p_{1}, p_{2}$ be the projections on the first and on the second factor of $S\times S$ respectively.
Then
\begin{equation}
\label{eq:conj}
(p_{1}^{*}a)\cdot (p_{2}^{*}a')=(p_{1}^{*}a')\cdot (p_{2}^{*}a) \hbox{ in } A^{4}(S\times S).
\end{equation}
\end{conj}

\begin{rem}
To ease the notation, we use the following convention: $a\times a':=(p_{1}^{*}a)\cdot (p_{2}^{*}a')$. So \eqref{eq:conj} becomes:
\[
a \times a'= a' \times a \in A^{4}(S\times S).
\]
\end{rem}

There are few examples in which Conjecture \ref{conj:Voi} has been verified (see  \cite{voisin1996remarks}, \cite{laterveer2016some}, \cite{laterveer2018algebraic}, \cite{2016arXiv161108821L}), but it is still open for a general $K3$ surface.
 There are some examples in which a generalization of this conjecture for surfaces with geometric genus greater than one is true (see \cite{Laterveer_2018}, \cite{MR3941135}).
There is also an analogous version of the conjecture for higher dimensional varieties, which is studied in \cite{voisin1996remarks},\cite{laterveer2016somed}, \cite{2017arXiv170806092L}, \cite{2017arXiv170806094L}, \cite{2017arXiv170600472B}, \cite{2017arXiv171203070L}, \cite{2018arXiv180300857V}, \cite{2018arXiv181011084B}. 
 
 Our aim is to present a new example in which Conjecture \ref{conj:Voi} is true, namely a family of Todorov surfaces.
Todorov surfaces were introduced by Todorov to provide counterexamples to Local and Global Torelli (\cite{Todorov_1981}). They were classified by Morrison (\cite{morrison1988moduli}) according to the \emph{fundamental invariants} $(\alpha,k)$, where the 2--torsion group of $\Pic(S)$ has order $2^{\alpha}$ and $k=8+K^{2}_{S}$. With this classification Morrison proves that there are exactly 11 non--empty irreducible families of Todorov surfaces corresponding to
\begin{multline*}
(\alpha,k)\in \lbrace(0,9),(0,10),(0,11),(1,10),(1,11),(1,12),\\
(2,12),(2,13),(3,14),(4,15),(5,16)\rbrace.
\end{multline*}
Todorov surfaces of type (0,9) are also known as \emph{Kunev surfaces}.\\

Conjecture \ref{conj:Voi} has been proven by Laterveer for the family of Todorov surfaces of type (0,9)(\cite{laterveer2016some}). Laterveer also proved the conjecture for the family of Todorov surfaces of type (1,10) (\cite{laterveer2018algebraic}). For both of these families the core of the proof was that an explicit description as complete intersections of the family was available. 

Moreover, the following result allows the reduction to the case of a double cover of $\Pj^{2}$ ramified along the union of two cubics, for which Conjecture \ref{conj:Voi} has been proven by Voisin (\cite[Theorem 3.4]{voisin1996remarks}).
\begin{teo}[Rito \cite{rito2009note}]
\label{teo:Rito}
Let $S$ be a Todorov surface and let $M$ be the $K3$ surface associated to $S$, i.e.
the the smooth minimal model of $S/\sigma$. Then there exists a generically finite degree--2 cover $M\to \Pj^{2}$ ramified along the union of two cubics.
\end{teo}
So any Todorov surface has an associated K3 surface for which Conjecture \ref{conj:Voi}  holds. To conclude that the conjecture holds for all Todorov surfaces one needs to be able to relate $0$--cycles on $S$ to $0$--cycles on $M$. 
The technique used to prove the conjecture in the known cases is based on Voisin's principle of ``spreading of cycles'' (\cite{voisin2013generalized}, \cite[Ch. 4]{voisin2014chow}).
This approach works as long as the irreducible family of Todorov surfaces considered $\mathcal{S}\to B$ has a nice enough description to prove the following property
\begin{equation}
\label{eq:prop}
	A_{\mathrm{hom}}^{2}(\mathcal{S}\times_{B}\mathcal{S})_{\mathbb{Q}}=0.
\end{equation}
The lack of such an explicit description for the other families of Todorov surfaces prevents to apply the same method to them.

We focus on the family of Todorov surfaces with fundamental invariants $(\alpha,k)=(2,12)$. We present an explicit description for this family as quotients of the complete intersection of four quadrics in $\Pj^{6}$. Our main result is the following theorem.

\begin{reftheorem}{teo:Tod}
Let $S$ be a general Todorov surface with fundamental invariants $(\alpha,k)=(2,12)$.\\
Then property \eqref{eq:prop} holds, and hence Conjecture \ref{conj:Voi} is true for $S$.
\end{reftheorem}

As proved by Laterveer in \cite{laterveer2018algebraic}, the conjecture has some interesting consequences on the motives. In particular, we prove that the transcendental part of the motives of the Todorov surface of type (2,12) and its associated K3- surface are isomorphic.
Up to now, the only types of K3 surfaces whose motives are known to be finite dimensional are the ones whose Picard number is high enough (\cite[Theorem 2]{MR2898747}) and such motives are of abelian type. We prove that the motive of the Todorov surfaces of type $(2,12)$ is of abelian type if the Picard number is big enough.

In Section \ref{sec:desc} we give an explicit description of the family studying the universal cover of the surfaces. To do so, we use the similar ideas to \cite{bini2014new} and \cite{neves2011unprojection}. In Section \ref{sec:cycles} we focus on 0--cycles by exploiting the idea of realizing the fibered self--product of the family of surfaces as a Zariski open set of a variety with trivial Chow groups. In Section \ref{sec:proof}  we prove Theorem \ref{teo:Tod} applying the ``spreading'' of algebraic cycles on a family following the approach of Laterveer (\cite{laterveer2018algebraic}) and Voisin (\cite{voisin2013generalized}, \cite{voisin2015generalized}). In Section \ref{sec:mot} we give a motivic version of the main result and some applications, following the approach in \cite{laterveer2018algebraic}.

\begin{notation}
We work on the field of complex numbers $\mathbb{C}$. A \emph{variety} is a quasi--projective separated scheme of finite type over $\mathbb{C}$ with the Zariski topology. A \emph{subvariety} is a reduced equidimensional subscheme.  A \emph{curve} is a variety of dimension one, a \emph{surface} is a variety of dimension two.\\
We denote the \emph{geometric genus} of a projective surface $S$ by
\[
p_g(S):=\dim H^{0}(S,\Omega^{2})=h^{2,0}(S).
\]
We denote the \emph{irregularity} of a projective surface as 
\[
q(S):=\dim(H^{0}(S,\Omega^{1}))=h^{1,0}(S).
\]
We denote the \emph{Euler-Poincar\'{e} characteristic} of a projective surface as 
\[
\chi(S):=\sum_{i=0}^{2}(-1)^{i}h^{i}(S, \mathcal{O}_{S})=1-q(S)+p_{g}(S).
\]
For $d\geq 0$, we denote the $d$th \emph{plurigenerus} of $S$ as
\[
	P_{d}(S):=h^{0}(dK_{S}).
\] 
If $X$ is a smooth $n$--dimensional variety, we denote by $A_{j}(X)=A^{n-j}(X)$ the Chow group of $j$--dimensional algebraic cycles modulo rational equivalence. When considering Chow groups with rational coefficients we use the following notation
\[
A_{j}(X)_{\mathbb{Q}}:=A_{j}(X)\otimes_{\mathbb{Z}} \mathbb{Q}.
\]
To denote algebraic cycles homologically trivial we use the notation $A^{j}_{\hom}(X)$, this is the kernel of the cycle class map $\gamma\colon A^{j}(X)\to H^{2j}(X,\mathbb{Z})$. 
Similarly, we denote by $A^{j}(X)_{AJ}$ the kernel of the Abel--Jacobi map:
\[
	AJ\colon A^{j}(X)\longrightarrow J^{2k-1}(X),
\]
where $J^{2k-1}(X)$ is the $k$--th Intermediate Jacobian of $X$.\\
We denote a projective point in $\Pj^{6}$ with homogeneous coordinates as $x:=(x_{0}\colon \dots\colon x_{6})$.
\end{notation}

\section{Geometry}
\subsection{Explicit description of the family}
\label{sec:desc}
\begin{defn}
\label{defn:todorov}
A \emph{Todorov surface} is a smooth projective surface $S$ of general type,  with $p_{g}(S)=1$, $q(S)=0$ and such that the bicanonical map $\phi_{|2K_{S}|}$ factors as
\[
\phi_{|2K_{S}|}\colon S \xrightarrow{\sigma} S \dashrightarrow \mathbb{P}^{r},
\]
where $\sigma\colon S\to S$ is an involution such that $S/\sigma$ is birational to a smooth $K3$ surface. We call the minimal resolution of $S/\sigma$ the \emph{K3 surface associated to $S$}.
\end{defn}

\begin{defn}
We define the following involutions on $\mathbb{P}^{6}$:
\begin{equation}
\label{eq:act}
\begin{split}
\sigma_{1}&\colon (x_{0}\colon \dots \colon x_{6})\mapsto (x_{0} \colon -x_{1}\colon -x_{2}\colon -x_{3}\colon -x_{4}\colon x_{5}\colon x_{6}),\\
\sigma_{2}&\colon (x_{0}\colon \dots \colon x_{6})\mapsto (x_{0} \colon x_{1}\colon x_{2}\colon -x_{3}\colon -x_{4}\colon -x_{5}\colon -x_{6}),\\
\sigma_{1}\circ \sigma_{2} &\colon (x_{0}\colon \dots \colon x_{6})\mapsto (x_{0} \colon -x_{1}\colon -x_{2}\colon x_{3}\colon x_{4}\colon -x_{5}\colon -x_{6}),\\
\sigma&\colon (x_{0}\colon \dots \colon x_{6})\mapsto (x_{0} \colon -x_{1}\colon -x_{2}\colon -x_{3}\colon -x_{4}\colon -x_{5}\colon -x_{6}).
\end{split}
\end{equation}
We define the group $G=<\sigma_{1},\sigma_{2}> \subset \Aut(\mathbb{P}^{6})$.
\end{defn}
\begin{rem}
Formulas \eqref{eq:act} describe an action of $G\cong(\mathbb{Z}/_{2\mathbb{Z}})^{2}$ on $H^{0}(\mathbb{P}^{6},\mathcal{O}_{\mathbb{P}^{6}}(1)))$, and therefore on $H^{0}(\mathbb{P}^{6},\mathcal{O}_{\mathbb{P}^{6}}(d)))$ for any $d\in \mathbb{N}$, which is compatible with the action of $G$ on $\mathbb{P}^{6}$.

For this action we have:
\begin{equation}
\label{eq:desc}
W:=H^{0}(\mathbb{P}^6,\mathcal{O}_{\Pj^{6}}(2))^{G}
=\langle x_{0}^{2},x_{1}^{2},x_{2}^{2},x_{3}^{2},x_{4}^{2},x_{5}^{2}, x_{6}^{2},x_{1}x_{2},x_{3}x_{4},x_{5}x_{6}\rangle_{\mathbb{C}}\cong \mathbb{C}^{10}.
\end{equation}
\end{rem}

\begin{defn}
We define 
\[
	\widetilde{U}\subset \Grass(4,W)/_{\mathrm{GL}(7,\mathbb{C})^{G}}
\]
to be the open set that parametrizes the complete intersection of four quadrics $V=\bigcap_{i=0}^{3}Q_{i}$ with  $Q_{0},\dots,Q_{3}\in H^{0}(\mathbb{P}^6,\mathcal{O}_{\Pj^{6}}(2))^{G}$. 
\end{defn}

\begin{rem}
We are considering all the four--dimensional subspaces in $W\cong\mathbb{C}^{10}$, so we are taking four quadrics in $W=H^{0}(\mathbb{P}^6,\mathcal{O}_{\Pj^{6}}(2))^{G}$ that are linearly independent. Then we are quotienting by
\[
	GL(7,\mathbb{C})^{G}=\{f\in GL(7) \hbox{ such that } \forall g\in G \, f\circ g=g\circ f\}.
\]
Since we can see $G$ as a subgroup of $GL(7,\mathbb{C})$, we can also consider the subgroup of the invariants, i.e. $GL(7,\mathbb{C})^{G}$ which acts naturally on $H^{0}(\mathbb{P}^{6},\mathcal{O}_{P}^{6}(1))$ with basis $(x_{0},\dots, x_{6})$. So we have an induced action of $GL(7,\mathbb{C})^{G}$ on $H^{0}(\mathbb{P}^6,\mathcal{O}_{\Pj^{6}}(d))$ for any $d\in \mathbb{N}$. In particular, we can consider its induced action  on $H^{0}(\mathbb{P}^6,\mathcal{O}_{\Pj^{6}}(2))$. Thus we have an action of $GL(7,\mathbb{C})^{G}$ on $W$, which induces an action of $GL(7,\mathbb{C})^{G}$ on $\Grass(4,W)$.
\end{rem}

\begin{defn}
\label{defn:Grass}
We define $U\subset \Grass(4,W)/_{\mathrm{GL}(7)^{G}}$ to be the open set that contains only the smooth complete intersections 
\[
	V=\bigcap_{i=0}^{3}Q_{i} \text{ with } Q_{0},\dots,Q_{3}\in H^{0}(\mathbb{P}^6,\mathcal{O}_{\Pj^{6}}(2))^{G}
\]
such that $V\cap \mathrm{Fix}_{G}=\varnothing$, so that the action of $G$ on $V$ is free.
So we consider the following situation: $\mathcal{U}\overset{p}{\rightarrow} U\subset  \Grass(4,W)/_{\mathrm{GL}(7)^{G}}$, where 
\[
	\mathcal{U}:=\left\{\Bigl([Q_{0}, Q_{1}, Q_{2}, Q_{3}],x\Bigr)\in U\times 
	\mathbb{P}^{6}\colon x\in \bigcap_{i=0}^{3}Q_{i} \right\}
	\subset U\times \mathbb{P}^{6},
\]
and $V=V_{u}\cong p^{-1}(u)=\{u\}\times V_{u}$ for some $u\in U$.\\
We define $S:=V/G$ to be the quotient under the action of $G$. 
\end{defn}

\begin{rem}
We have the following numerical situation
\[
\begin{cases*}
	\chi(V)=8,\\
	q(V)=0,\\
	p_{g}(V)=\chi(V)-1+q(V)=7,\\
	K^{2}_{V}=16.
\end{cases*}
\]
We notice that $S$ is smooth because the action of $G$ is free, and by \cite[Lemma VI.11, Examples VI.12]{beauville1996complex} we get: 
\[
	K^{2}_{S}=4, q(S)=0, \chi(\mathcal{O}_{S})=2, p_{g}(S)=1.
\]
\end{rem}

\begin{rem}
In Definition \ref{defn:Grass}, we ask that $V\cap\Fix(G)=\varnothing$. This is possible since the fixed locus has dimension one. Indeed we have that each involution fixes a $\mathbb{P}^{2}$ and a $\mathbb{P}^{3}$, in particular:
	\begin{itemize}
		\item $\Fix(\sigma_{1})=\mathbb{P}^{2}_{x_{0},x_{5},x_{6}}\sqcup 
			\mathbb{P}^{3}_{x_{1},x_{2},x_{3},x_{4}}$
		\item $\Fix(\sigma_{2})=\mathbb{P}^{2}_{x_{0},x_{1},x_{2}}\sqcup 
			\mathbb{P}^{3}_{x_{3},x_{4},x_{5},x_{6}}$
		\item $\Fix(\sigma_{1}\circ \sigma_{2})=\mathbb{P}^{2}_{x_{0},x_{3},x_{4}}
		\sqcup \mathbb{P}^{3}_{x_{1},x_{2},x_{5},x_{6}}$
	\end{itemize} 
	\begin{align*}
	\Rightarrow \Fix(G)&=\Fix(\sigma_{1})\cap\Fix(\sigma_{2})\cap 
	\Fix(\sigma_{1}\circ\sigma_{2})\\
	&=\{(1:0:0:0:0:0:0)\}\sqcup 
	\mathbb{P}^{1}_{x_{5},x_{6}}\sqcup\mathbb{P}^{1}_{x_{1},x_{2}}\sqcup
	\mathbb{P}^{1}_{x_{3},x_{4}}\sqcup \mathbb{P}^{1}_{x_{5},x_{6}}.
	\end{align*}
\end{rem}


\begin{prop}
\label{prop:desc}
Let $G\cong(\mathbb{Z}/_{2\mathbb{Z}})^{2}$ be the group
of automorphisms of $\mathbb{P}^{6}$ acting as in \eqref{eq:act} and let $V$ be a smooth complete intersection of four quadrics $Q_{0},Q_{1},Q_{2},Q_{3}\in W=H^{0}(\mathbb{P}^{6},\mathcal{O}_{\Pj^{6}}(2))^{G}$ contained in $U$.
 Then the quotient surface $V/G$ is a Todorov surface of type $(2,12)$.
\end{prop}
In order to prove this result, we need to prove that the involution $\sigma$ on $V/G$ is such that the quotient is a (singular) K3 surface so that the bicanonical map of $V/G$ factors through it.\\
Since $\sigma$ acts as the identity on $W$ and, in particular, on the equation defining $V$, 
we have that $\sigma \in \Aut(V)$. Moreover, $\sigma$ commutes with $G$, so we can consider its action on the quotient, i.e. $\sigma[p]=[\sigma(p)]$ is well defined for any $[p]\in V/G=S$.

Then, Proposition \ref{prop:desc} follows directly from the following result.

\begin{lem}
The quotient surface $S/\sigma=(V/G)/\sigma$ is a K3 surface with at most nodes as singularities  and the bicanonical map of $S$ factors through it.
\end{lem}

\begin{proof}
We consider the bicanonical maps of $V$ and $V/G$. Since $V$ is a complete intersection surface, we have that $h^{1}(2K_{V} ) = h^{1}(\mathcal{O}_{V} (2))=0$. 
Since $K_{V}$ is ample, by Serre's duality we get also $h^{2}(2K_{V})=0$. In particular, by Riemann--Roch Theorem, we have
\[
	P_{2}(V):=h^{0}(2K_{V})=\chi(\mathcal{O}_{V}(2K_{V}))=K^{2}_{V}+\chi(\mathcal{O}_{V})=24.
\]
So the bicanonical map is $\phi_{2K_{V}}\colon V\to \mathbb{P}^{23}$.\\

Since $K_{V} \cong \mathcal{O}_{V} (1)$ is ample and $q \colon V\to S$ is finite and ètale, then $K_{S}$ is ample by Nakai-Moishezon criterion \cite[Ch. V, Theorem 1.10]{hartshorne2013algebraic}) and $q^{*}K_{S} \cong K_{V}$. So by Mumford’s vanishing theorem $h^{1}(2K_{S}) = 0$.
Since $h^{0}(K_{S}) = p_{g}(S) > 0$ and $K_{S}$ is ample, by Serre duality then $P_{2}(S)=6$ by Riemann-Roch theorem.

We have the following commutative diagram:
\[
	\begin{tikzcd}
		V\arrow[r,"\phi_{2K_{V}}"] \arrow[d, "q"] \arrow[drr,dashrightarrow, "\psi"']
		& \mathbb{P}^{23}\arrow[d,dashrightarrow]\arrow[dr,dashrightarrow]\\
		S=V/G \arrow[r,"\phi_{2K_{S}}"']
		&\mathbb{P}^{5}\arrow[r,hookrightarrow]
		&\mathbb{P}^{9}.
	\end{tikzcd}
\]
Then it holds that
\[
	H^{0}(2K_{V/G})=H^{0}(2K_{V})^{G}=\langle x_{0}^{2},\dots, x_{6}^{2},x_{1}x_{2},x_{3}x_{4},x_{5}x_{6}\rangle_{\mathbb{C}}\mod H^{0}(I_{V}(2)).
\]
It is convenient to look at the bicanonical image in $\mathbb{P}^{9}$, so we study the map 
\begin{align*}
	\psi\colon \mathbb{P}^{6}&\longrightarrow \mathbb{P}^{9}\\
	(x_{0}\colon \dots \colon x_{6})&\mapsto(x_{0}^{2}\colon \dots \colon x_{6}^{2}\colon x_{1}x_{2}\colon x_{3}x_{4}\colon x_{5}x_{6}).
\end{align*}
The map $\psi$ is given by the chosen monomial quadrics. Since $V$ is the complete intersection of four quadrics in this system, the restrictions of the quadrics of $W$ to $V$ are elements of $H^{0}(2K_{V} )^{G}\cong \mathbb{C}^{6}$ which is isomorphic to $H^{0}(2K_{V/G})$. So, we get that $\psi(V)\subset \mathbb{P}^{5}$, and this $\mathbb{P}^{5}$ is defined by the 4 linear equations in $\mathbb{P}^{9}$ given by those quadrics defining $V$.\\


We consider the subgroup $H:=\langle\sigma,\sigma_{1},\sigma_{2}\rangle\cong \left(\mathbb{Z}/_{2\mathbb{Z}}\right)^{3}$ of automorphisms of $\mathbb{P}^{6}$ of order 8 and we notice that the involution $\sigma$ commutes with $G=\langle\sigma_{1},\sigma_{2}\rangle$. Since $W$ (hence the equations defining $V$ ) is $\sigma$--invariant, then it is also $H$--invariant, so $H$ is a group of automorphisms of $V$. Moreover, since the action on $\mathbb{P}^9$ is deduced from the action on $\mathbb{P}^6$, we have that $\psi$ is $H$--equivariant, and so is its restriction $\psi\vert_{V}\colon V\to \mathbb{P}^{5}$. Next we notice that $H$ acts trivially on $\mathbb{P}^{9} \cong \mathbb{P}(W^{*})$, hence on $\mathbb{P}^{5}\cong \mathbb{P}((W/H^{0}(I_{V}(2)))^{*})$. Thus the equivariant map $\psi\vert_{V}$ factors through the quotient by $H\colon V\to V/H\to \mathbb{P}^{5} /H\cong \mathbb{P}^{5}$ .

A standard direct computation of the fibers shows that $\psi$ is finite of degree 8. It follows that $\psi$ is the quotient by $H$. Then we have a commutative diagram
\[
	\begin{tikzcd}
		V \arrow[r,  "\psi"]\arrow[d,  "q"]&V/H\\
		V/G\arrow[ur,  "r"]
	\end{tikzcd}
\]
where $r$ is the quotient by the involution induced by $\sigma$.\\
We want to show that $\psi(V )$ is a K3 surface with ordinary double points. That it has only ordinary double points as singularities follows from the fact that it is a quotient of a 
smooth surface by an involution. In particular, we claim that it is a complete intersection of three quadrics in $\mathbb{P}^{5}$. Indeed, in such a case we have that by Adjunction formula 
\begin{equation}
\label{eq:can}
	K_{\psi(V)}=\mathcal{O}_{\psi(V)}(-6+2+2+2)=\mathcal{O}_{\psi(V)}.
\end{equation}
Moreover, $\psi(V ) = V/H$ being a complete intersection, we will have $h^{1}(\mathcal{O}_{\psi(V )}) = 0$,  and so $\psi(V)$ is a K3 surface. \\

Now let us prove that $\psi(V)$ is a complete intersection of three quadrics in $\mathbb{P}^{5}$.
To ease the notation, we name the coordinates in $\mathbb{P}(H^{0}(2K_{V}/G))$ as
\[
	z_{0}=x_{0}^{2},\dots, z_{6}=x_{6}^{2}, z_{12}=x_{1}x_{2},z_{34}=x_{3}x_{4},z_{56}=x_{5}x_{6}.
\]
Then the image of $\psi\colon \mathbb{P}^{6}\to \mathbb{P}^{9}$ has dimension 6 and 
\begin{equation*}
	\psi(\mathbb{P}^{6})=\{(z_{0},\dots,z_{6},z_{12},z_{34},z_{56})\in \mathbb{P}^{9}\colon
	 z_{12}^{2}=z_{1}z_{2}, \, z_{34}^{2}=z_{3}z_{4}, z_{56}^{2}=z_{5}z_{6}\}.
\end{equation*}
Indeed, $\psi(\mathbb{P}^{6})$ is contained in this locus. Since  the intersection of these three quadrics defines an irreducible 6--dimensional variety which is complete intersections, this is indeed $\psi(\mathbb{P}^{6})$. When we restrict to $V$, we get that $\psi(V)$ is a complete intersection of three quadrics and four linear forms in $\mathbb{P}^{9}$ given by the four quadrics defining $V$. 
\end{proof}

\begin{prop}
\label{prop:dim}
The family of Definition \ref{defn:Grass} of Todorov surfaces of type (2,12) is 12--dimensional.
\end{prop}
In order to prove this result we need a preliminary lemma on the dimension of the stabilizer.

\begin{lem}
\label{lem:stab}
A generic point in the Grassmanian $\Grass(4, W )$ has a 1--dimensional stabilizer.
\end{lem}

\begin{proof}
First of all, we notice that a multiple of the identity matrix $\lambda I$ with $\lambda \in \mathbb{C}^{*}$ acts trivially. So the generic stabilizer has dimension greater or equal to one. Since $\{x\in \Grass(4, W ), \,\dim\left(\mathrm{stab}_{GL(7)^{G}} (x)\right) \geq l\}$ are Zariski closed subsets for $l \geq 0$, it is enough to find a point in the Grassmanian which has 1–dimensional stabilizer to prove the claim.\\
Let us consider the point in $\Gr(4,10)$ given by the following four quadrics in $H^{0}(\mathbb{P}^6,\mathcal{O}_{\Pj^{6}}(2))^{G}$:
\begin{align*}
	Q_{0}&=\{(x_{0},\dots, x_{6})\in \mathbb{P}^{6} 
		\colon x_{0}^{2}+x_{1}^{2}+x_{3}^{2}+x_{5}^{2}=0\};\\	
	Q_{1}&=\{(x_{0},\dots, x_{6})\in \mathbb{P}^{6} 
		\colon x_{2}^{2}+x_{4}^{2}+x_{6}^{2}=0\};\\
	Q_{2}&=\{(x_{0},\dots, x_{6})\in \mathbb{P}^{6} 
		\colon x_{1}x_{2}+x_{3}x_{4}+x_{5}^{2}=0\};\\
	Q_{3}&=\{(x_{0},\dots, x_{6})\in \mathbb{P}^{6} 
		\colon x_{3}x_{4}+x_{5}x_{6}=0\}.
\end{align*}
Since any element of a finite subgroup of $GL(7, \mathbb{C})$ is diagonalizable,
asking to commute with the group $G$ for an element $f\in \mathrm{GL}(7)$ is equivalent to ask for $f$ to preserve the eigenspaces.
Let us denote $G\cong (\mathbb{Z}/_{2\mathbb{Z}})^{2}= \mathbb{Z}/_{2\mathbb{Z}}e_{1}\oplus\mathbb{Z}/_{2\mathbb{Z}}e_{2}, 
	G^{*}\cong \mathbb{Z}/_{2\mathbb{Z}}\epsilon_{1}\oplus \mathbb{Z}/_{2\mathbb{Z}}\epsilon_{2}$.
When we look at the action of $G$ on $H^{0}(\mathcal{O}_{\mathbb{P}^{6}}(1))$, we have a decomposition into irreducible components on the characters: $W_{\chi_{0}}\oplus 2W_{\chi_{\epsilon_{1}}}\oplus 2W_{\chi_{\epsilon_{2}}}\oplus 2W_{\chi_{\epsilon_{1}+\epsilon_{2}}}$. Then, $\{x_{0}\}$ generates $W_{\chi_{0}}$, $\{x_{1},x_{2}\}$ is a basis of $2W_{\chi_{\epsilon_{1}}}$, $\{x_{3},x_{4}\}$ is a basis of $2W_{\chi_{\epsilon_{2}}}$ and $\{x_{5},x_{6}\}$ is a basis of $2W_{\chi_{\epsilon_{1}+\epsilon_{2}}}$.\\
So a general $M\in \mathrm{GL}(7,\mathbb{C})^{G}$ is a matrix of the type
\begin{equation}
\label{eq:matrix1}
	M=\begin{pmatrix}
	a&0&\hdotsfor{4}&0\\
	0&b&c&0&\hdotsfor{2}&0\\
	0&d&e&0&\hdotsfor{2}&0\\
	0&\dots&0&f&g&0&0\\
	0&\dots&0&h&i&0&0\\
	0&\hdotsfor{3}&0&l&m\\
	0&\hdotsfor{3}&0&n&r\\
	\end{pmatrix},
\end{equation}
for some coefficients $a,b,c,d,e,f,g,h,i,l,m,n,r \in \mathbb{C}$.\\
Expand the matrix $M$ in the centralizer of $G$ in $\mathrm{GL}(W)$ according to the character decomposition of $W$ as $M = (a, M_{1}, M_{2}, M_{3})$. 
Then considering the action of M consecutively on $Q_{1}, Q_{3}, Q_{2}, Q_{0}$ one finds that $M$ fixes $Q_{i}$ up to a non-zero constant, all upper diagonal element are zero, all lower diagonal elements of $M_{2}$ and $M_{3}$ are zero, and finally $M_{2} = \lambda I$. So the projective automorphism group of the generic intersection of quadrics is a finite group.

\end{proof}

%

\begin{proof} of Proposition \ref{prop:dim}.
 Let $[V ], [V ′] \in U$ be two complete intersections such that 
$V /G \cong V ′ /G$. Then this isomorphism lifts to the universal covers 
$V \cong V'$ . Since we embed the universal covers in $\mathbb{P}^{6}$ via the 
canonical map, such isomorphism extends to an automorphism of 
$\mathbb{P}^{6}$ . As a conclusion, two distinct points of $U$ give rise to distinct 
isomorphism classes of Todorov surfaces and one can see $U$ as a subset 
of the moduli space of Todorov surfaces.
In order to compute the dimension of the family we are describing, we compute the dimension of the base
\[
		U\subset \Grass(4,W)/_{\mathrm{GL}(7,\mathbb{C})^{G}}.
\]
We have that the dimension of the Grassmanian variety is:
\[
	\dim \Grass(4,W)=4(10-4)=24.
\]
Let us compute now $\dim \mathrm{GL}(7,\mathbb{C})^{G}$.

We have that
\[
	\dim \mathrm{GL(7)^{G}}=1+3\cdot 2^{2}=13.
\]
We notice, however, that the action of $\mathrm{GL}(7)^{G}$ is not faithful.

By Lemma \ref{lem:stab}, we conclude that 
\begin{equation}
\label{eq:dim}
	\dim(U)=\dim \Grass(4,W)/_{ \mathrm{GL}(7)^{G}}=24-13+1=12.
\end{equation}
\end{proof}

So we have found a 12-dimensional family of Todorov surfaces of type $(2,12)$, whose general element is $S=V/G$, where $V$ is a smooth complete intersection of four linearly independent quadrics in $\mathbb{P}^{6}$ which are $G$--invariant. \\

We are finally able to prove our main result to describe the family of Todorov surfaces of type $(2,12)$.

\begin{teo}
\label{teo:desc}
Let $S$ be a general Todorov surface with fundamental invariants $(\alpha, k) = (2, 12)$.
Then the canonical model of $S$ is a quotient surface $V/G$ where $V$ is the smooth complete intersection of four independent quadrics $Q_{0},Q_{1},Q_{2},Q_{3}\in H^{0}(\mathbb{P}^{6},\mathcal{O}_{\Pj^{6}}(2))^{G}$ contained in $U$, and $G\cong(\mathbb{Z}/_{2\mathbb{Z}})^{2}$ is the group of automorphisms of $\mathbb{P}^{6}$ acting as above.
Conversely, any such surface $V/G$ is a Todorov surface of type $(2,12)$.
\end{teo}

\begin{proof}\emph{of Theorem \eqref{teo:desc}}.\\
By Proposition \ref{prop:desc} it follows that $V/G$ is a Todorov surface of type (2,12). 
In order to prove the first part of the theorem, we use a dimensional argument. Since the number of moduli of the family of Todorov surfaces of type (2,12) is 12,  and the family is irreducible (see \cite{Todorov_1981}, \cite[Theorem 7.5]{morrison1988moduli} and \cite[Remark 5.3.5]{usui1991mixed}, \cite[Section 4.2]{lee2015deformations}), by \eqref{eq:dim} we conclude that we are describing a Zariski open set of an irreducible moduli space, so we are describing the general element of the family.
\end{proof}

\section{Cycles}
\label{sec:cycles}

By means of Theorem \ref{teo:desc}, we can give an explicit description of the family of Todorov surfaces of type $(2,12)$. Now we want to introduce a natural compactification of the parameter space for the family of Todorov surfaces of this type, which is more useful when dealing with cycles. 
\begin{defn}
\label{defn:family}
We define 
\begin{equation}
\label{eq:B}
\overline{B}:=\prod_{i=0}^{3}\Pj\left(H^{0}(\mathbb{P}^6,\mathcal{O}_{\Pj^{6}}(2))^{G}\right)\cong \Pj^{9}\times\Pj^{9}\times\Pj^{9}\times\Pj^{9}.
\end{equation}
Let $\mathcal{V}\overset{p}{\rightarrow} B$ denote the total space of the family of the complete intersections $\bigcap_{i=0}^{3}Q_{i}\subset \Pj^{6}$, where $b\in B$ and $B\subseteq \overline{B}$ is a Zariski open set which parametrizes only the smooth intersections, i.e. $\overline{B}$ is the projective closure of $B$.\\
We are in the following situation:
\[
	\mathcal{V}:=\left\{\Bigl([Q_{0}], [Q_{1}], [Q_{2}], [Q_{3}],x\Bigr)\in B\times 
	\mathbb{P}^{6}\colon x\in \bigcap_{i=0}^{3}Q_{i} \right\}
	\subset B\times \mathbb{P}^{6}.
\]
For any $b=([Q_{0}],[Q_{1}],[Q_{2}],[Q_{3}])\in B$, we define 
\[
	V_{b}:=\bigcap_{i=0}^{3}Q_{i}\cong p^{-1}(b)=\{b\}\times V_{b}.
\]
In particular, the morphism $p$ corresponds to the first projection of $B\times \mathbb{P}^{6}$ restricted to $\mathcal{V}$.
Since the action of $G\cong(\mathbb{Z}/_{2\mathbb{Z}})^{2}$ on $\overline{B}\times \mathbb{P}^{6}$ is non trivial only on the second component, we can consider its action on $\mathcal{V}$ and we get $\mathcal{S}:=\mathcal{V}/G\to B$. 
In particular, $\mathcal{S}$ is the family of canonical models of Todorov surface above $B$ and Theorem \ref{teo:desc} shows that a general Todorov can be parametrized by $B \twoheadrightarrow U$, so $B$ is a parameter space for general Todorov surfaces of type $(2,12)$. 
\end{defn}

The situation is the following
\[
			\begin{tikzcd}
			\mathcal{V}\arrow[d,"\pi"' ] &\subset B\times \mathbb{P}^{6}
			&&\mathcal{S}\arrow[d,"\overline{\pi}"' ] 
			&\subset B\times \mathbb{P}^{6}/G\\
			B &&&B,
			\end{tikzcd}
\]
and these are the tautological families. Then $\mathcal{V}$ is a smooth quasi--projective variety 
since $B$ is smooth and the fibers of $V\to B$ are smooth and of the same dimension (see \cite[Theorem 3.3.27]{schoutens2010use}).

%
%


By definition \ref{defn:todorov}, to each Todorov surface $S_{b}=V_{b}/G$ we have two associated $K3$ surfaces, one is $\overline{M_{b}}=S_{b}/\sigma$, and the other is its resolution of singularities $M_{b}=(\overline{M_{b}})^{res}$.

\begin{rem}
\label{rem:sing}
We notice that $\overline{M_{b}}$ is a singular variety with quotient singularities. In general, Chow groups of singular varieties do not admit intersection product or a ring structure. But, in our case, the identification  $A_{*}(\overline{M})_{\mathbb{Q}}\cong (A_{*}(S)_{\mathbb{Q}})^{\sigma}$ allows to work with Chow rings as usual (see \cite[Example 8.3.12]{fulton2013intersection}).\\
Moreover, if a surface $M$ is singular with finitely many quotient singularities 
with resolution $M^{res}$, $A_{0}(M)_{\mathbb{Q}}$ is canonically isomorphic to $A^{2}(M)_{\mathbb{Q}}$, the operational Chow cohomology (see \cite{fulton2013intersection}). By Kimura’s exact sequence, there is an isomorphism
$A^2(M)_{\mathbb{Q}} \cong A^2(M^{res})_{\mathbb{Q}}$ induced by pullbacks.
Then $A_{0}(M )\cong A_{0}(M^{\mathrm{res}} )$  and, since we are only interested in zero--cycles on surfaces, we can pretend that surfaces with rational double points as singularities are smooth.
\end{rem}
By the previous remark we can reduce our situation to the following diagram
\[
			\begin{tikzcd}
			\mathcal{V}\arrow[d,"4:1", "q"' ]\\
			\mathcal{S}=\mathcal{V}/G\arrow[d,"2:1", "f"' ]\\
			\mathcal{M},
			\end{tikzcd}
			\qquad\qquad
			\begin{tikzcd}
			G \curvearrowright V_{b} \qquad\arrow[d, "4:1", "q_{b}"' ]\\
				\qquad \sigma \curvearrowright S_{b}=V_{b}/G 
				\arrow[d, "2:1", "f_{b}"' ]\\
			M_{b}.
		\end{tikzcd}
\]
where $\mathcal{M}$ parametrizes  the smooth $K3$ surfaces associated to the Todorov surfaces, and on the right we have the fiberwise description.

\subsection{0--cycles}
In order to prove that the family of Todorov surfaces of type $(2,12)$ verifies Voisin's conjecture \ref{conj:Voi}, we want to move the problem on the associated $K3$ surfaces.The key result allowing to show Voisin’s conjecture is the isomorphism
\begin{equation}
\label{eq:key}
A^{2}_{hom}(S)_{\mathbb{Q}}\cong A^{2}_{hom}(M)_{\mathbb{Q}},
\end{equation}
where $S$ is a Todorov surface of type $(2,12)$ and $M$ is its associated $K3$ surface. We prove the isomorphism \eqref{eq:key} in Section \ref{sec:proof} (see Theorem \ref{teo:cyc}). We first show how to deduce Voisin’s conjecture for the family of Todorov surfaces of type $(2,12)$ as a corollary of \eqref{eq:key}. The proof follows the one given in \cite[Corollary 3.2]{laterveer2018algebraic}.

\begin{teo}
\label{teo:Tod}
Let $S$ be a general Todorov surface with fundamental invariants $(\alpha,k)=(2,12)$.\\
Then Conjecture \ref{conj:Voi} is true for $S$.
\end{teo}

\begin{proof} 
First of all we notice that it is enough to prove the theorem with rational coefficients. Indeed, by Rojtman's Theorem (\cite{rojtman1980torsion}) there is no torsion in $A^{4}_{\mathrm{hom}}(S\times S)$.\\
Let $M$ be the associated $K3$ surface to $S$, i.e. the minimal resolution of $S/\sigma$. We have a commutative diagram:
\[
		\begin{tikzcd}
			A^{2}_{\mathrm{hom}}(S)_{\mathbb{Q}}\otimes A^{2}_{\mathrm{hom}}(S)_{\mathbb{Q}}
			\arrow[r ] 
			& A^{4}_{\mathrm{hom}}(S\times S)_{\mathbb{Q}}\\
			A^{2}_{\mathrm{hom}}(M)_{\mathbb{Q}} \otimes A^{2}_{\mathrm{hom}}(M)_{\mathbb{Q}}
			\arrow[u] \arrow[r]
			&A^{4}_{\mathrm{hom}}(M\times M)_{\mathbb{Q}}. \arrow[u]\\
		\end{tikzcd}
\]
By \eqref{eq:key}, the left vertical arrow is an isomorphism. We recall that by Rito's result (Theorem \ref{teo:Rito}) the K3 surface can be described as the blow-up of a double cover of $\mathbb{P}^{2}$ ramified along the union of two cubics. By \cite[Theorem 3.4]{voisin1996remarks}, Conjecture \ref{conj:Voi} is then true for $M$, i.e. 
\[
a \times a'= a' \times a \in A^{4}(M\times M) \quad \forall a,a'\in A^{2}_{\mathrm{hom}}(M).
\]
Hence the conjecture holds for $S$ too.
\end{proof}

\begin{rem}
\label{rem:gen}
We notice that Theorem \ref{teo:Tod} holds for all Todorov surfaces of type (2,12). Indeed, by \cite[Lemma 3.2]{voisin2014chow}, it is enough to prove Voisin's conjecture \ref{conj:Voi} for the general (or even the very general by  \cite[Remark 3.3]{voisin2014chow}) member of the family of Todorov surfaces of type (2,12). The upshot of this approach is that for the general member of this family we have an explicit description in terms of complete intersections and quotients by Theorem \ref{teo:desc}.
\end{rem}

The proof of \eqref{eq:key} will be shown to be a consequence of the following property
\begin{equation}
\label{eq:core}
	A^{2}_{hom}(\mathcal{S}\times_{B}\mathcal{S})_{\mathbb{Q}}=
	A^{2}_{hom}(\mathcal{V}\times_{B}\mathcal{V})_{\mathbb{Q}}=0.
\end{equation}
\begin{rem}
\label{rem:core}
As the map $g\colon \mathcal{V}\overset{4:1}{\longrightarrow} \mathcal{S}=\mathcal{V}/G$ is a finite surjective morphism, $A^{2}_{\mathrm{hom}}(\mathcal{V}\times_{B}\mathcal{V})_{\mathbb{Q}}=0$ implies $A^{2}_{hom}(\mathcal{S}\times_{B}\mathcal{S})_{\mathbb{Q}}=0$. So it is enough to prove the statement for the fibered product of the complete intersections family $\mathcal{V}\times_{B}\mathcal{V}$.
\end{rem}

\subsection{Small Chow groups}

The proof of \eqref{eq:core} is based on the results in \cite[Proposition 4.5]{laterveer2018algebraic}, \cite{voisin2015generalized}, \cite{voisin2013generalized}. 
The idea is to see the fiber product $\mathcal{V}\times_{B}\mathcal{V}$ as a Zariski--open set of a variety $X$ whose Chow groups are trivial. 
\begin{defn}\cite[4.3]{voisin2014chow}
We say that a variety $X$ has \emph{trivial Chow groups} if the cycle class maps are injective, i.e.
\[
	A_{i}(X)_{\mathbb{Q}}\hookrightarrow H_{2i}(X,\mathbb{Q}) \quad \forall i.
\]
\end{defn}
\begin{rem}
In \cite[4.3]{voisin2014chow} the definition of trivial Chow groups is given in the case of a smooth variety using singular cohomology groups, but we can adapt such a  definition in the case of a singular variety using homology instead.
\end{rem}Examples of smooth projective varieties with trivial Chow groups are toric varieties, projective spaces and varieties stratified by affine spaces.

\begin{defn}
\label{defn:X}
We recall that $\overline{B}:=\prod_{i=0}^{3}\Pj\left(H^{0}(\mathbb{P}^6,\mathcal{O}_{\Pj^{6}}(2))^{G}\right)$ is the projective closure of $B$, so that $B$ is a Zariski open set which parametrizes the smooth complete intersections.
    We define the variety $X\subset \overline{B}\times \mathbb{P}^{6}\times \mathbb{P}^{6}$ as
    \begin{equation}
    \label{eq:X}
    	X:=\Bigl\{\Bigl(([Q_{0}]:[Q_{1}]:[Q_{2}]:[Q_{3}]),(p,q)\Bigr)\in \overline{B}\times  
    	\mathbb{P}^{6}\times \mathbb{P}^{6} : 
	Q_{i}(p)=Q_{i}(q)=0\, \forall i\Bigr\}.
    \end{equation}
    Then $X$ contains the fiber product $\mathcal{V}\times_{B}\mathcal{V}$ as a Zariski open set. 
\end{defn}

    Identity \eqref{eq:core} follows by proving that $X$ as in Definition \ref{defn:X} has trivial Chow groups (see Proposition \ref{prop:trivial}). 
\begin{prop}
\label{prop:key}
Suppose that $B\subset \overline{B}$ is small enough to have a smooth morphism $\mathcal{V}\to B$.
Then $A^{2}_{hom}(\mathcal{V}\times_{B}\mathcal{V})_{\mathbb{Q}}=0$, which implies $A^{2}_{hom}(\mathcal{S}\times_{B}\mathcal{S})_{\mathbb{Q}}=0$ by Remark \ref{rem:core}. 
\end{prop}

\begin{proof}
Let $X$ be the variety defined in \eqref{eq:X} and assume it has trivial Chow groups (see Proposition \ref{prop:trivial}).
Let $D:=X\backslash (\mathcal{V}\times_{B}\mathcal{V})$ be the boundary divisor and $m:=\dim X$. Let $a\in A^{2}_{\mathrm{hom}}(\mathcal{V}\times_{B}\mathcal{V})_{\mathbb{Q}}$ a homologically trivial cycle. So $a$ is the restriction of a cycle in $X$, i.e. there exists $\overline{a}\in A_{m-2}(X)_{\mathbb{Q}}$ such that $\overline{a}|_{\mathcal{V}\times_{B}\mathcal{V}}=a$ and $[\overline{a}]|_{\mathcal{V}\times_{B}\mathcal{V}}=0\in H^{4}(\mathcal{V}\times_{B}\mathcal{V},\mathbb{Q})$. Performing a resolution of singularities on $X$
\[
	\begin{tikzcd}
		\widetilde{X}\arrow[r, "res"]&X\\
		\widetilde{D}	\arrow[u,hook]\arrow[r,"res","r"' ]&D	\arrow[u,hook, "i"']
	\end{tikzcd}
\]
 we find out that the class $[\overline{a}]$ comes from a Hodge class $\beta\in H^{2}(\widetilde{D},\mathbb{Q})$ since $\overline{a}\in A_{m-2}(X)_{\mathbb{Q}}$ and since it is homologically trivial on $\mathcal{V}\times_{B}\mathcal{V}$.
By Lefschetz Theorem on $(1,1)$--classes, we have that $\beta$ is algebraic, so there exists a cycle $b\in A^{1}(\widetilde{D})_{\mathbb{Q}}$ such that $[b]=\beta$. Let us define $\overline{\overline{a}}:=\overline{a}-i_{*}(r_{*}b)\in A^{\mathrm{hom}}_{m-2}(X)_{\mathbb{Q}}$, which is a trivial group since $X$ has trivial Chow groups; thus
$\overline{\overline{a}}=0$ in  $A^{\mathrm{hom}}_{m-2}(X)_{\mathbb{Q}}$ which yields $0=\overline{\overline{a}}|_{\mathcal{V}\times_{B}\mathcal{V}}=a$ and we conclude that $A^{2}_{\mathrm{hom}}(\mathcal{V}\times_{B}\mathcal{V})=0$. 
\end{proof}

It is left to prove that $X$ has trivial Chow groups, following  the argument in \cite{laterveer2018algebraic}.
    
We consider the projection
    \begin{equation}
    \label{eq:pi}
    	X\overset{\pi}{\rightarrow} \mathbb{P}^{6}\times\mathbb{P}^{6},
    \end{equation}
    then the fiber over a point is a product of projective spaces 
    \[
    	\pi^{-1}(p,q)\cong \{b\in \overline{B}\colon Q_{i}(p)=Q_{i}(q)=0\}\cong 
    	\mathbb{P}^{r}\times \mathbb{P}^{r}\times \mathbb{P}^{r}\times 
    	\mathbb{P}^{r}\subset \overline{B},
    \]
   for some $r\leq9$, but the fiber does not have constant dimension on the whole space.\\

To prove that $X$ has trivial Chow groups, the idea is to find a stratification of $\mathbb{P}^{6}\times\mathbb{P}^{6}$ such that the fiber of $\pi$ has constant dimension on each stratum. In Section \ref{subsec:strat} we give such a stratification, and in Section \ref{subsec:trivial} we prove that $X$ has trivial Chow groups.

\subsubsection{Stratification of $\mathbb{P}^{6}\times\mathbb{P}^{6}$}
\label{subsec:strat}

\begin{defn}
For any $j,k\in \{0,\dots,6\}$ we define partial diagonals as follows
\begin{equation*}
\begin{split}
	\Delta^{j,k}_{\pm}&:=\Bigl\{(p,q)\in \mathbb{P}^{6}\times\mathbb{P}^{6} : 
	\exists \lambda \in \mathbb{C}^{*} \hbox{ s.t. } q_{0}=\pm \lambda p_{0},
	q_{j}=-\lambda p_{j},\\ 
	&\qquad\qquad\qquad\qquad\qquad\qquad\qquad q_{k}=-\lambda 
	p_{k}, q_{i}=\lambda p_{ i} \, \forall i\neq j,k\Bigr\};\\
	 \Delta^{j,k,l,m}_{\pm}&:=\Bigl\{(p,q)\in \mathbb{P}^{6}\times\mathbb{P}^{6} : \exists 
	 \lambda \in \mathbb{C}^{*} \hbox{ s.t. } q_{0}=\pm \lambda p_{0},
	 q_{i}=-\lambda p_{i} \, \\ 
	 &\qquad\qquad\qquad\qquad\forall i\in\{ j,k,l,m\} \hbox{ and } 
	q_{i}=\lambda p_{ i} \, \forall i\neq j,k,l,m  \Bigr\};\\
	\Delta^{0}&:=\Bigl\{(p,q)\in \mathbb{P}^{6}\times\mathbb{P}^{6} : \exists \lambda 
	\in \mathbb{C}^{*} \hbox{ s.t. } q_{0}=\lambda p_{0}, q_{i}=-\lambda p_{i} 
	\forall i\neq0 \Bigr\}\,.
\end{split}
\end{equation*}
\end{defn}

\begin{rem}
Since we are in a projective space, some of the partial diagonals are identified by the -1 product. Indeed we have: $\Delta^{1,2,3,4}_{+}=\Delta^{5,6}_{-}$;
$\Delta^{1,2,3,4}_{-}=\Delta^{5,6}_{+}$;
$\Delta^{1,2,5,6}_{+}=\Delta^{3,4}_{-}$;
$\Delta^{1,2,5,6}_{-}=\Delta^{3,4}_{+}$;
$\Delta^{3,4,5,6}_{+}=\Delta^{1,2}_{-}$;
 $\Delta^{3,4,5,6}_{-}=\Delta^{1,2}_{+}$.

\end{rem}

\begin{lem}
\label{lem:strat}
There exists a Zariski-open set of points $U\subset \mathbb{P}^{6}\times \mathbb{P}^{6}$ such that, for any $(p,q)\in U$, the fiber of the projection $\pi\colon X\to \mathbb{P}^{6}\times \mathbb{P}^{6}$ is $\pi^{-1}(p,q)\cong \mathbb{P}^{7}\times\mathbb{P}^{7}\times\mathbb{P}^{7}\times\mathbb{P}^{7}$. 
Moreover, considering the complement 
$Z=(\mathbb{P}^{6}\times\mathbb{P}^{6})\backslash U$ with
\[
	Z:=\Delta_{\mathbb{P}^{6}\times \mathbb{P}^{6}}\cup 
	\Delta^{1,2}_{+}\cup\Delta^{1,2}_{-}\cup\Delta^{3,4}_{+}\cup\Delta^{3,4}_{-}\cup
	\Delta^{5,6}_{+}\cup\Delta^{5,6}_{-}\cup\Delta^{0},
\]
for all the points $(p',q')\in Z$ the fiber is $\pi^{-1}(p',q')\cong \mathbb{P}^{8}\times\mathbb{P}^{8}\times\mathbb{P}^{8}\times\mathbb{P}^{8}$.
\end{lem}

\begin{proof}
Each point of $\mathbb{P}^{6}$ imposes one condition on each component $\mathbb{P}(H^{0}(\mathbb{P}^{6}, \mathcal{O}_{\mathbb{P}^{6}}(2))^{G})$ of 
\[
\overline{B}=\Pi_{i=0}^{3}\left(\mathbb{P}(H^{0}(\mathbb{P}^{6}, \mathcal{O}_{\mathbb{P}^{6}}(2))^{G})\right)\cong \mathbb{P}^{9}\times\mathbb{P}^{9}\times\mathbb{P}^{9}\times\mathbb{P}^{9}.
\]
Indeed, let us consider $Q\in \mathbb{P}(H^{0}(\mathbb{P}^{6}, \mathcal{O}_{\mathbb{P}^{6}}(2))^{G})\cong \mathbb{P}^{9}$. By \eqref{eq:desc}, we can make the condition $Q(p)=0$ explicit as
\[
	Q(p)=\alpha p_{0}^{2}+\beta p_{1}^{2}+\gamma p_{2}^{2}+\delta p_{3}^{2}+\epsilon  p_{4}^{2}+\zeta p_{5}^{2}+\eta p_{6}^{2}+\theta p_{1}p_{2}+\kappa p_{3}p_{4}+\lambda p_{5}p_{6}=0,
\]
with $p=(p_{0}:\dots:p_{6})\in \mathbb{P}^{6}$ and $\alpha,\dots,\lambda \in \mathbb{C}$. So given a point $(p,q)\in \mathbb{P}^{6}\times\mathbb{P}^{6}$ we have two such conditions, which we can represent by a matrix
\[
	A(p,q):=\begin{pmatrix}
		p_{0}^{2}&p_{1}^{2}&p_{2}^{2}&p_{3}^{2}&p_{4}^{2}&p_{5}^{2}&p_{6}^{2}&p_{1}p_{2}&p_{3}p_{4}&p_{5}p_{6}\\
		q_{0}^{2}&q_{1}^{2}&q_{2}^{2}&q_{3}^{2}&q_{4}^{2}&q_{5}^{2}&q_{6}^{2}&q_{1}q_{2}&q_{3}q_{4}&q_{5}q_{6}\
	\end{pmatrix}.
\]
In general, $A$ has maximum rank, so that inside $\mathbb{P}^{6}\times\mathbb{P}^{6}$ there is a Zariski--open set of pair of points $(p,q)$, such that the conditions imposed by $p$ and $q$ on the the linear system $\vert W\vert$ are independent. So for a general point in $\mathbb{P}^{6}\times\mathbb{P}^{6}$ we have that the fiber is
\[
 	\pi^{-1}(p,q)\cong \left\{b\in \overline{B}\colon Q_{i}(p)=Q_{i}(q)=0\right\}\cong 
    	\mathbb{P}^{7}\times\mathbb{P}^{7}\times\mathbb{P}^{7}\times\mathbb{P}^{7} \subset \overline{B},
\]
However, the rank of $A$ is not always maximum. 
there is a locus $Z\subset \mathbb{P}^{6}\times\mathbb{P}^{6}$ where the conditions imposed by $p$ and $q$ are no more independent.

For any point $(p,q)\in Z$ we have that the rank of $A$ is not maximum, so the fiber of such a point is
\[
	\pi^{-1}(p,q)\cong \left\{b\in \overline{B}\colon Q_{i}(p)=Q_{i}(q)=0\right\}\cong 
    	\mathbb{P}^{8}\times\mathbb{P}^{8}\times\mathbb{P}^{8}\times\mathbb{P}^{8} 
	\subset \overline{B}.
\] 

Indeed, if $\rank A(p, q) = 1$ then all the $2\times 2$ minors vanish. In particular the 	  	vanishing of the first 7 minors gives $0=p_{i}^{2}q_{j}^{2}−p_{j}^{2}q_{i}^{2} = 
	(p_{i}q_{j}−p_{j}q_{i})(p_{i}q_{j}+p_{j}q_{i}), \forall i,j$.
	For a given $i_{0}$, if $p_{i_{0}} =0$ then $q^{2}_{i_{0}}p^{2}_{j} =0$ for any $j$ 
	which, as $p\in \mathbb{P}^{6}$, yields $q_{i_{0}}=0$. So there is a $i_{0}$ such  
	that $p_{i_{0}}q_{i_{0}}\neq0$. Then we get $q_{j} = \epsilon_{j}\lambda p_{j}$ for
	 any $j$, with $\lambda = \frac{q_{i_{0}}}{p_{i_{0}}}$ and $\epsilon_{j}\in \{\pm1\}$. 
	 By the vanishing of all the $2\times2$ minors, we also get 
	 $0 = p^{2}_{i_{0}} q_{1}q_{2} − q^{2}_{i_{0}} p_{1}p_{2} = \lambda^{2}p_{i_{0}}^{2} 
	 p_{1}p_{2}(\epsilon_{1}\epsilon_{2} − 1)$ i.e. $\epsilon_{1} = \epsilon_{2}$.
	Likewise, we get $\epsilon_{3} =\epsilon_{4}$ and $\epsilon_{5} =\epsilon_{6}$. 
	Conversely for any pair $(p,q)\in \mathbb{P}^{6}\times \mathbb{P}^{6}$
	satisfying those conditions, a direct computation shows that $\rank A(p, q) = 1$ i.e. 
	$(p, q) \in Z$. 	\\
We define $U\colon=\left(\mathbb{P}^{6}\times \mathbb{P}^{6}\right) \backslash Z$. 
Then we claim that $U$ is the Zariski open subset over which the fiber has the lowest dimension, i.e. for any $(p,q)\in U$ we have that $\pi^{-1}(p,q)\cong \mathbb{P}^{7}\times \mathbb{P}^{7}\times \mathbb{P}^{7}\times \mathbb{P}^{7}$. \\
Indeed, by definition, we have that $U = (\mathbb{P}^{6}\times \mathbb{P}^{6})\backslash Z$ is the locus $\left\{(p, q)\in \mathbb{P}^{6}\times\mathbb{P}^{6} \text{ s.t. } \rank A(p, q) 	=2\right\}$.
\end{proof}

\subsubsection{$X$ has trivial Chow group}
\label{subsec:trivial}
In order to prove that $X$ has trivial Chow groups, we introduce a related property following the idea given in \cite{laterveer2018algebraic}.
\begin{defn}
(\cite{totaro2014chow})Let $V$ be a quasi--projective variety, and let $A_{i}(V,j)$ denote  Bloch's higher Chow groups. Then there are functorial cycle class maps
\[
	A_{i}(V,j)\to \mathrm{Gr}^{W}_{-2i}H_{2i+j}(V,\mathbb{Q}),
\]
where $W$ denotes the Deligne's weight filtration on the Borel--Moore homology (see \cite[Appendix B]{peters2008mixed} and \cite{deligne1974poids}). We recall that, since $W$ is an increasing filtration, the associated graded piece is 
\[
	\mathrm{Gr}^{W}_{-2i}H_{2i+j}(V,\mathbb{Q}):=
	\frac{W_{-2i}H_{2i+j}(V,\mathbb{Q})}{W_{-2i-1}H_{2i+j}(V,\mathbb{Q})}.
\]
\begin{itemize}[leftmargin=3.5cm]
	\item[\textbf{Weak property:}] we say that $V$ has the weak property if there are 		isomorphisms induced by the cycle class maps
		\[
			A_{i}(V)_{\mathbb{Q}}\overset{\sim}{\rightarrow}W_{-2i}
			H_{2i}(V,\mathbb{Q}) \quad 
			\forall i.
		\]
	\item[\textbf{Strong property:}] we say that $V$ has the strong property if it has the 
	weak property and there are surjections induced by the cycle class maps 
	\[
		A_{i}(V,1)_{\mathbb{Q}}\twoheadrightarrow \mathrm{Gr}^{W}_{-2i}
		H_{2i+1}(V,\mathbb{Q}) 
		\quad \forall i.
	\]
\end{itemize}

\end{defn}
\begin{rem}
We notice that we have the following implications: strong property $\Rightarrow$ weak property $\Rightarrow$ trivial Chow groups. Indeed we have
\[
	A_{i}(V)_{\mathbb{Q}}\overset{\sim}{\rightarrow}W_{-2i}
	H_{2i}(V,\mathbb{Q}) \hookrightarrow H_{2i}(V,\mathbb{Q}).
\]
\end{rem}

We have the following useful result.
\begin{lem}\cite[Lemma 4.2,4.3,4.4]{laterveer2018algebraic}
\label{lem:res}
Let $X$ be a quasi--projective variety.
\begin{enumerate}
	\item Let $Y\subset X$ be a closed subvariety and $U=X\backslash Y$ be its 	
	complement. If $Y$ and $U$ have the strong property, then $X$ has the strong 
	property too.
	\item Suppose that $X$ admits a stratification by strata of the form 
	$\mathbb{A}^{k}\backslash L$, where $L$ is a finite union of linearly embedded 
	affine subspaces. Then $X$ has the strong property.
	\item If $X$ has the strong property and $P\to X$ is a projective bundle, then $P$
 	has the strong property too.
 \end{enumerate}
\end{lem}

This result is proven using extended localization exact sequences for Higher Chow groups and long exact sequence of relative cohomology.
 It will allow us, through an adequate partition of $X$, to prove the following:
\begin{prop}
\label{prop:trivial}
$X$ has the strong property. In particular it has trivial Chow groups, i.e.
\[
	A_ {*}^{\mathrm{hom}}(X)_{\mathbb{Q}}=0.
\]	
\end{prop}
\begin{proof}
  For any point $(p,q)\in \mathbb{P}^{6}\times \mathbb{P}^{6}$ we can consider the fiber of the projection  $X\overset{\pi}{\rightarrow} \mathbb{P}^{6}\times\mathbb{P}^{6}$ as in \eqref{eq:pi}. We have that
    \[
    	\pi^{-1}(p,q)\cong \{b\in \overline{B}\colon Q_{i}(p)=Q_{i}(q)=0\}\cong 
    	\mathbb{P}^{r}\times \mathbb{P}^{r}\times \mathbb{P}^{r}\times 
    	\mathbb{P}^{r}\subset \overline{B},
    \]
   for some $r\leq9$, where $\overline{B}$ is defined as in \eqref{eq:B}.
The dimension of the fiber is not constant, nonetheless, by Lemma \ref{lem:strat}, we have that it is constant on a Zariski open set $U\subset \mathbb{P}^{6}\times \mathbb{P}^{6}$ and on its complement $Z$.\\
Our situation is the following
\[
	\begin{tikzcd}
		X_{U}\arrow[r,hook]&X \arrow[r, hookleftarrow]\arrow[d, "\pi"]& X_{Z}\\
		U\arrow[r,hook]&\mathbb{P}^{6}\times \mathbb{P}^{6} 
		\arrow[r, hookleftarrow]& Z,
	\end{tikzcd}
\]
where $X_{Z}=\pi^{-1}(Z)$, $X_{U}=\pi^{-1}(U)$, $U$ and $Z$ are defined as in Lemma \ref{lem:strat}.
We want to prove that $X_{Z}$ and $X_{U}$ have the strong property and then conclude by means of Lemma \ref{lem:res}(1).\\

For every $i=0,\dots, 6$ we define 
\begin{gather*}
	A_{i}:=\{(p,q)\in \mathbb{P}^{6}\times \mathbb{P}^{6}\colon 
	p_{i}\neq 0 \hbox{ and } q_{i}\neq 0\};\\
	B_{i}:=\{(p,q)\in \mathbb{P}^{6}\times \mathbb{P}^{6}\colon 
	p_{i}= 0 \hbox{ and } q_{i}= 0\};\\
	C_{i}:=\{(p,q)\in \mathbb{P}^{6}\times \mathbb{P}^{6}\colon 
	p_{i}\neq 0 \hbox{ and } q_{i}= 0\};\\
	D_{i}:=\{(p,q)\in \mathbb{P}^{6}\times \mathbb{P}^{6}\colon  p_{i}=0 \text{ and } q_{i}\neq 0\}
\end{gather*}
First of all, we consider the locus $Z$. The intersections $Z\cap C_{0}$ and $Z\cap D_{0}$ are empty, whereas $\overline{A}_{0}:=Z\cap A_{0}$ is not. Since $A_{0} \cap \Delta_{\mathbb{P}^{6}\times \mathbb{P}^{6}}=A_{0}\cap \Delta_{0}$, we have that $\overline{A}_{0}:=Z\cap A_{0}=A_{0}\cap \Delta_{\mathbb{P}^{6}\times \mathbb{P}^{6}}\cup 
	\Delta^{1,2}_{+}\cup\Delta^{1,2}_{-}\cup\Delta^{3,4}_{+}\cup\Delta^{3,4}_{-}\cup
	\Delta^{5,6}_{+}\cup\Delta^{5,6}_{-}$ , so
$\overline{A}_{0}$ is isomorphic to $7$ copies of $\mathbb{A}^{6}$ via the map
\[
	\begin{tikzcd}
	\left(\left[1: \frac{p_{1}}{p_{0}}:\frac{p_{2}}{p_{0}}:\frac{p_{3}}{p_{0}}:
	\frac{p_{4}}{p_{0}}:\frac{p_{5}}{p_{0}}:\frac{p_{6}}{p_{0}}\right],
	\left[\pm\lambda: \pm\lambda\frac{p_{1}}{p_{0}}:\pm\lambda\frac{p_{2}}{p_{0}}:
	\pm\lambda\frac{p_{3}}{p_{0}}:\pm\lambda\frac{p_{4}}{p_{0}}:
	\pm\lambda\frac{p_{5}}{p_{0}}:\pm\lambda\frac{p_{6}}{p_{0}}\right]\right) 
	\arrow[d,mapsto]\\
	\left(\pm\frac{p_{1}}{p_{0}},\pm\frac{p_{2}}{p_{0}},\pm\frac{p_{3}}{p_{0}},\pm\frac{p_{4}}{p_{0}},\pm\frac{p_{5}}{p_{0}},\pm\frac{p_{6}}{p_{0}}\right),
		\end{tikzcd}
\]
with $\lambda\in \mathbb{C}^{*}$. For the intersection $\overline{Z}_{0}:=Z\cap B_{0}$, we can consider $\overline{A}_{1}:=\overline{Z}_{0}\cap A_{1}$. We notice that $\Delta^{1,2}_{+}\cap B_{0}=\Delta^{1,2}_{-}\cap B_{0}$, $\Delta^{3,4}_{+}\cap B_{0}=\Delta^{3,4}_{-}\cap B_{0}$ and $\Delta^{5, 6}_{+}\cap B_{0}=\Delta^{5, 6}_{-}\cap B_{0}$. Moreover, $A_{1}\cap B_{0}\cap \Delta_{\mathbb{P}^{6}\times \mathbb{P}^{6}}=A_{1}\cap B_{0}\cap \Delta_{0}$, so $\overline{A}_{1}$ is isomorphic to $4$ copies of $\mathbb{A}^{5}$ via the map
\[
	\begin{tikzcd}
	\left(\left[0: 1:\frac{p_{2}}{p_{1}}:\frac{p_{3}}{p_{1}}:\frac{p_{4}}{p_{1}}:\frac{p_{5}}{p_{1}}:\frac{p_{6}}{p_{1}}\right],\left[0: \pm\lambda:\pm\lambda\frac{p_{2}}{p_{1}}:\pm\lambda\frac{p_{3}}{p_{1}}:\pm\lambda\frac{p_{4}}{p_{1}}:\pm\lambda\frac{p_{5}}{p_{1}}:\pm\lambda\frac{p_{6}}{p_{1}}\right]\right) 
	\arrow[d,mapsto]\\
	\left(\pm\frac{p_{2}}{p_{1}},\pm\frac{p_{3}}{p_{1}},\pm\frac{p_{4}}{p_{1}},\pm\frac{p_{5}}{p_{1}},\pm\frac{p_{6}}{p_{1}}\right),
		\end{tikzcd}
\]
with $\lambda\in \mathbb{C}^{*}$. 
The intersections $\overline{Z_{0}}\cap C_{1}$ and $\overline{Z_{0}}\cap D_{1}$ are empty and next we can consider $\overline{Z}_{1}:=\overline{Z}_{0}\cap B_{1}$. Iterating this process we get
\begin{gather*}
	\begin{cases}
		\overline{Z}_{-1}:=Z \\
		\overline{Z}_{i}:=Z\cap \left(\bigcap_{j=0}^{i}B_{j}\right) &\text{ for }  i\in\{0, 
		\dots 6\}.
	\end{cases}
\end{gather*}
Summarizing the resulting iterating process, we have that $\overline{A}_{0}\cong \coprod^{7}\mathbb{A}^{6}$; $\overline{A}_{1}\cong \coprod^{4}\mathbb{A}^{5}$; $\overline{A}_{2}\cong \coprod^{3}\mathbb{A}^{4}$, since at this level intersecting with $\Delta_{\mathbb{P}^{6}\times \mathbb{P}^{6}}$ is the same as intersecting also with $\Delta^{1,2}_{\pm}$; $\overline{A}_{3}\cong \coprod^{3}\mathbb{A}^{3}$; $\overline{A}_{4}\cong \coprod^{2}\mathbb{A}^{2}$, since here intersecting with $\Delta^{1,2}_{\pm}$ is the same as intersecting also with $\Delta^{3,4}_{\pm}$; $\overline{A}_{5}\cong \coprod^{2}\mathbb{A}^{1}$ and $\overline{A}_{6}\cong \mathbb{A}^{0}$.\\
Thus we have a stratification of $Z$ by affine strata as $Z=\coprod_{i=0}^{6} \overline{A}_{i} $.\\
	
We consider now $U=\left(\mathbb{P}^{6}\times\mathbb{P}^{6}\right)\backslash Z$. Then $U_{0}:=U\cap A_{0}$ is isomorphic to $\mathbb{A}^{12}$ minus $7$ copies of $\mathbb{A}^{6}$ via the map
\[
	\begin{tikzcd}
	\left(\left[1: \frac{p_{1}}{p_{0}}:\frac{p_{2}}{p_{0}}:\frac{p_{3}}{p_{0}}:
	\frac{p_{4}}{p_{0}}:\frac{p_{5}}{p_{0}}:\frac{p_{6}}{p_{0}}\right],
	\left[1: \frac{q_{1}}{q_{0}}:\frac{q_{2}}{q_{0}}:
	\frac{q_{3}}{q_{0}}:\frac{q_{4}}{q_{0}}:
	\frac{q_{5}}{q_{0}}:\frac{q_{6}}{q_{0}}\right]\right) 
	\arrow[d,mapsto]\\
	\left(\frac{p_{1}}{p_{0}},\frac{p_{2}}{p_{0}},\frac{p_{3}}{p_{0}},\frac{p_{4}}{p_{0}},\frac{p_{5}}{p_{0}},\frac{p_{6}}{p_{0}}, \frac{q_{1}}{q_{0}},\frac{q_{2}}{q_{0}},\frac{q_{3}}{q_{0}},\frac{q_{4}}{q_{0}},\frac{q_{5}}{q_{0}},\frac{q_{6}}{q_{0}}\right).
		\end{tikzcd}
\]
Iterating the process as above, we get
\begin{gather*}
	\begin{cases}
		T_{-1}:=U \\
		T_{i}:=U\cap \left(\bigcap_{j=0}^{i}B_{j}\right)\\						
		U_{i}:=T_{i-1}\cap A_{i}
	\end{cases} \quad \text{ for }  i\geq 0;
%
\end{gather*}
In particular we get that $U_{0}\cong \mathbb{A}^{12}\backslash \coprod^{7}\mathbb{A}^{6}$, $U_{1}\cong \mathbb{A}^{10}\backslash \coprod^{4}\mathbb{A}^{5}$, $U_{2}\cong \mathbb{A}^{8}\backslash \coprod^{3}\mathbb{A}^{4}$, $U_{3}\cong \mathbb{A}^{6}\backslash \coprod^{3}\mathbb{A}^{3}$, $U_{4}\cong \mathbb{A}^{4}\backslash \coprod^{2}\mathbb{A}^{1}$, $U_{5}\cong \mathbb{A}^{2}\backslash \coprod^{2}\mathbb{A}^{1}$.

Next we need to stratify $V_{0}:=U\cap C_{0}$ and $W_{0}:=U\cap D_{0}$, in order to obtain affine strata. \\
\begin{figure}[!htb]
\centering
\begin{tikzpicture}[grow=right,transform shape, scale=0.67, level distance=4cm]
\Tree 
[.$V_{0}\cap B_{1}$
	[.$\cap B_{2}$
		[.$\cap B_{3}$
			[.$\cap B_{4}$
				[.$\cap B_{5}$
				]
				[.$\cap C_{5}$
				]
				[.$\cap A_{5}$
				]
				[.$\cap D_{5}$
				]
			]
			[.$\cap C_{4}$
			[.$\cap B_{5}$
				]
				[.$\cap C_{5}$
				]
				[.$\cap A_{5}$
				]
				[.$\cap D_{5}$
				]
			]
			[.$\cap A_{4}$
			]
			[.$\cap D_{4}$
			]
		]
		[.$\cap C_{3}$
		[.$\cap B_{4}$
				[.$\cap B_{5}$
				]
				[.$\cap C_{5}$
				]
				[.$\cap A_{5}$
				]
				[.$\cap D_{5}$
				]
			]
			[.$\cap C_{4}$
			[.$\cap B_{5}$
				]
				[.$\cap C_{5}$
				]
				[.$\cap A_{5}$
				]
				[.$\cap D_{5}$
				]
			]
		]
		[.$\cap A_{3}$
		]
		[.$\cap D_{3}$
		]
	]
	[.$\cap C_{2}$
		[.$\cap B_{3}$
		[.$\cap B_{4}$
				[.$\cap B_{5}$
				]
				[.$\cap C_{5}$
				]
				[.$\cap A_{5}$
				]
				[.$\cap D_{5}$
				]
			]
			[.$\cap C_{4}$
			[.$\cap B_{5}$
				]
				[.$\cap C_{5}$
				]
				[.$\cap A_{5}$
				]
				[.$\cap D_{5}$
				]
			]
		]
		[.$\cap C_{3}$
		[.$\cap B_{4}$
				[.$\cap B_{5}$
				]
				[.$\cap C_{5}$
				]
				[.$\cap A_{5}$
				]
				[.$\cap D_{5}$
				]
			]
			[.$\cap C_{4}$
			[.$\cap B_{5}$
				]
				[.$\cap C_{5}$
				]
				[.$\cap A_{5}$
				]
				[.$\cap D_{5}$
				]
			]
		]
		[.$\cap A_{3}$
		]
		[.$\cap D_{3}$
		]
	]
	[.$\cap A_{2}$
	]
	[.$\cap D_{2}$
	]
]
] 
\end{tikzpicture}
\caption{Stratification of $V_{0}\cap B_{1}$}
\label{fig:graph}
\end{figure}

To stratify $V_{0}$ we proceed by steps, at the first step we consider the intersections $V_{0}\cap A_{1}\cong \mathbb{A}^{11}\backslash \{0\}\times \mathbb{A}^{10}$, $V_{0}\cap D_{1}\cong \mathbb{A}^{10}$, $V_{0}\cap C_{1}$ and $V_{0}\cap B_{1}$ which are not affine and require further stratifications. Iterating this process, 
at the $i$--th step we intersect with $A_{i+1}, B_{i+1}, C_{i+1}, D_{i+1}$ getting a tree structure as we can see in the graph in Fig. \ref{fig:graph}. Each time we intersect with $A_{i}$ or $D_{i}$ we get an affine strata, when we intersect with $B_{i}$ or $C_{i}$ we have to take a step further. The process end for each branch after five steps, when we intersect with $A_{5}, B_{5}, C_{5}$ and $D_{5}$. For example $V_{0}\cap B_{1}\cap B_{2}\cap C_{3}\cap B_{4}\cap D_{5}\cong \mathbb{A}^{3}\backslash \{0\}\times \mathbb{A}^{2}$.\\
We can apply the same process to stratify $W_{0}$.

So we can see $U=\coprod_{i=0}^{6}\left(U_{i}\coprod V_{i} \coprod W_{i}\right)$ as a disjoint union of varieties of type $\mathbb{A}^{k}\backslash L$, where $L$ is a finite union of linearly embedded affine spaces. By Lemma \ref{lem:res}(2), $U$ has the strong property and so does $Z$.
Since $X_{Z}=\pi^{-1}(Z)$ is a fibration over $Z$, whose fiber are product of projective spaces $\mathbb{P}^{8}\times \mathbb{P}^{8}\times \mathbb{P}^{8}\times \mathbb{P}^{8}$, then, by means of Lemma \ref{lem:res}(3), $X_{Z}$ has the strong property too. With the same argument, $X_{U}=\pi^{-1}(U)$ has the strong property since it is a fibration over $U$ with fiber $\mathbb{P}^{7}\times \mathbb{P}^{7}\times\mathbb{P}^{7}\times \mathbb{P}^{7}$. Then, by Lemma \ref{lem:res}(1) we conclude that $X$ has the strong property, and in particular it has trivial Chow groups.
\end{proof}

\subsection{Proof of \eqref{eq:key}}
\label{sec:proof}
It is left to prove the isomorphism \eqref{eq:key}, i.e. when dealing with homologically trivial 0--cycles on a Todorov surface, we can actually move the problem to the associated $K3$ surface to gain some more information. 


\begin{teo}
\label{teo:cyc}
Let $S$ be a Todorov surface with fundamental invariants $(\alpha,k)=(2,12)$.
Let $M$ be the associated $K3$ surface to $S$.
Then there is an isomorphism 
\[
A^{2}_{hom}(S)_{\mathbb{Q}}\cong A^{2}_{hom}(M)_{\mathbb{Q}}.
\]
\end{teo}

\begin{proof}
%

We look for a correspondence in $A^{2}(\mathcal{S}\times_{B}\mathcal{S})$ that is homologically trivial when restricted to each fiber .\\
Let $\Delta_{\mathcal{S}}\in A^{2}(\mathcal{S}\times_{B}\mathcal{S})$ denote the relative diagonal. We consider the following relative correspondence
\[
\Gamma=2\Delta_{\mathcal{S}}-^{t}\Gamma_{f}\circ \Gamma_{f}\in A^{2}(\mathcal{S}\times_{B}\mathcal{S}),
\]
where $\Gamma_{f}\subset \mathcal{S}\times \overline{\mathcal{M}}$ is the correspondence given by the graph of $f\colon \mathcal{S}=\mathcal{V}/G\to \overline{\mathcal{M}}$, and $^{t}\Gamma_{f}$ is the transpose correspondence. We denote the restriction to the fiber as $\Gamma_{b}:=\Gamma\vert_{S_{b}\times S_{b}}$.\\
Looking at the action induced by $\Gamma_{b}$ on cohomology we get
\[
(\Gamma_{b})_{*}=2\id_{H^{*}(S_{b})}-(f_{b})^{*}(f_{b})_{*}\colon H^{*}(S_{b},\mathbb{Q})\to H^{*}(S_{b},\mathbb{Q}).
\]
We claim that the action of $\Gamma_{b}$ is zero on $H^{2,0}(S_{b})$. 
This is true if and only if $(f_{b})^{*}(f_{b})_{*}=2\id$ on $H^{2,0}(S_{b})$. 
By \cite[Lemma 1]{Inose1979} $(f_{b})^{*}(f_{b})_{*}=(\Delta_{S_{b}})_{*}+\sigma_{*}$ (we point out that this holds also in the singular case by Remark \ref{rem:sing}).
If $\sigma_{*}$ acts as $-\id$ on $H^{2,0}(S_{b})$, then it implies that $H^{2,0}(\overline{M_b})=H^{2,0}(S_b/\sigma)=0$, but since $h^{2,0}(\overline{M}_{b})=1=h^{2,0}(S_{b})$, the involution $\sigma_{*}$ acts as the identity on $H^{2,0}(S_{b})$.
So we get $(f_{b})^{*}(f_{b})_{*}=(\Delta_{S_{b}})_{*}+\sigma_{*}=2\id_{H^{2,0}(S_{b})}$ and the claim is proved.\\
Now we consider the K\"{u}nneth decomposition of the diagonal of $S_{b}$
\[
\left[\Delta(S_{b})\right]=\sum_{i=0}^{4}\left[\pi_{i}^{b}\right]=\left[\pi^{b}_{0}\right]+\left[\pi^{b}_{2}\right]+\left[\pi^{b}_{4}\right]\in H^{4}(S_{b}\times S_{b},\mathbb{Q}),
\]
where $\left[\pi_{i}^{b}\right]\in H^{4-i}(S_{b},\mathbb{Q})\otimes H^{i}(S_{b},\mathbb{Q})\subset H^{4}(S_{b}\times S_{b},\mathbb{Q})$ is the $i$--th K\"{u}nneth component. The first and third components are zero, due to the fact that $q(S_{b})=h^{1,0}(S_{b})=0$.
Since the K\"{u}nneth conjecture $C(X)$ is known to be true for surfaces (\cite[ch. 3.1.1]{murre2013lectures}), we know that the K\"{u}nneth components are algebraic, i.e. they come from algebraic cycles $\pi^{b}_{i}\in A^{2}(S_{b}\times S_{b})_{\mathbb{Q}}$.

We recall that the action of $\pi_{i}^{b}$ in cohomology is the identity on $H^{i}(S_{b},\mathbb{Q})$ and it is zero elsewhere (\cite[Ch. 6.1]{murre2013lectures}). We are mainly interested in the second component $\pi^{b}_{2}=\Delta(S_{b})-\pi^{b}_{0}-\pi^{4}_{b}$, where $\pi^{b}_{0}=\{x\}\times S_{b}$, $\pi^{b}_{4}=S_{b}\times \{x\}$, and $x$ is a point in $S_{b}$.

\begin{rem}
We notice that $\pi^{b}_{2}$ exists also relatively
, i.e. there exists $\pi^{\mathcal{S}}_{2}=\Delta_{\mathcal{S}}-\pi_{0}-\pi_{4} \in A^{2}(\mathcal{S}\times_{B} \mathcal{S})$ such that for any $b\in B$ $\pi^{\mathcal{S}}_{2}\vert_{b}=\pi^{b}_{2}$ and $\pi_{i}|_{b}=\pi_{i}^{b}$ for any $i$. Indeed, let us consider the class of an ample divisor $h\in A^{1}(\mathbb{P}^{6})$, and its self--intersection $h^{2}=h\cdot h\in A^{2}(\mathbb{P}^{6})$. Next we consider $h^{2}\times B\in A^{2}(\mathbb{P}^{6}\times B)$, and its restriction to $\mathcal{V}\subset B\times \mathbb{P}^{6}$, i.e. $\overline{h}:=(h^{2}\times B)\vert_{\mathcal{V}}\in A^{2}(\mathcal{V})$. Looking at the fiber, we have that for any point $b\in B$
\[
	\overline{h}\vert_{V_{b}}=\{x_{0},\dots, x_{d}\},
\]
where $d=\deg V_{b}=16$. Then we define
\begin{gather*}
\pi_{0}^{\mathcal{V}}:=\frac{1}{d}pr_{1}^{*}\left(\overline{h}\vert_{\mathcal{V}}\right)\in A^{2}(\mathcal{V}\times_{B}\mathcal{V});\\
\pi_{4}^{\mathcal{V}}:=\frac{1}{d}pr_{2}^{*}\left(\overline{h}\vert_{\mathcal{V}}\right)\in A^{2}(\mathcal{V}\times_{B}\mathcal{V});
\end{gather*}
where $pr_{1},pr_{2}$ are the projections in the fiber product
\[
	\begin{tikzcd}
		\mathcal{V}\times_{B}\mathcal{V}\arrow[r, "pr_{2}"] \arrow[d,  "pr_{1}"']&
		 \mathcal{V}\arrow[d]\\
		\mathcal{V}\arrow[r]&B.
	\end{tikzcd}
\]
When we restrict to each fiber and we pass to cohomology, by the Künneth decomposition, we have
\begin{gather*}
	[\pi_{0}^{\mathcal{V}}]\vert_{V_{b}}=[p]\times [V_{b}]\in H^{4}(V_{b},\mathbb{Q})\otimes H^{0}(V_{b},\mathbb{Q}),\\
	[\pi_{4}^{\mathcal{V}}]\vert_{V_{b}}=[V_{b}]\times [p]\in H^{0}(V_{b},\mathbb{Q})\otimes H^{4}(V_{b},\mathbb{Q});
\end{gather*}
where $p\in V_{b}$ is a point.
So we can define the relative Künneth component of the diagonal $\pi^{\mathcal{V}}_{2}=\Delta_{\mathcal{V}}-\pi_{0}^{\mathcal{V}}-\pi_{4}^{\mathcal{V}} \in A^{2}(\mathcal{V}\times_{B} \mathcal{V})$. The we can use the push--forward of $\mathcal{V}\overset{q}{\rightarrow}\mathcal{S}=\mathcal{V}/G$ to get the relative Künneth component of $\mathcal{S}$: $\pi_{2}^{\mathcal{S}}=\deg(q)\Delta_{\mathcal{S}}-q_{*}\pi_{0}^{\mathcal{V}}-q_{*}\pi_{4}^{\mathcal{V}}$.
\end{rem}

\begin{lem}
\label{lem:psib}
The correspondence $\Psi_{b}$ defined as
\[
\Psi_{b}:=\Gamma_{b}\circ \pi^{b}_{2}=\left(2\Delta(S_{b})-^{t}\Gamma_{f_{b}}\circ \Gamma_{f_{b}}\right)\circ \pi^{b}_{2}\in A^{2}(S_{b}\times S_{b})_{\mathbb{Q}}.
\]
has its cohomology class in $H^{4}(S_{b}\times S_{b},\mathbb{Q})\cap \left(H^{1,1}(S_{b})\otimes H^{1,1}(S_{b})\right)$.
\end{lem}
 \begin{proof}
 By definition of $\pi^{b}_{2}$, when we look at the action in cohomology we have that $\Psi_{b}$ acts only on $H^{2}(S_{b}, \mathbb{Q})$. Moreover, since we proved that the action of $\Gamma_{b}$ is zero on $H^{2,0}(S_{b})$, we conclude that the class of $\Psi_{b}$ lives in $H^{1,1}(S_{b})\otimes H^{1,1}(S_{b})$.
 \end{proof}

%
By the previous remark, we can consider also the relative correspondence
\[
\Psi:=\Gamma\circ \pi^{\mathcal{S}}_{2}=(2\Delta_{\mathcal{S}}-^{t}\Gamma_{f}\circ \Gamma_{f})\circ \pi^{\mathcal{S}}_{2}\in A^{2}(\mathcal{S}\times_{B} \mathcal{S})_{\mathbb{Q}},
\]
where $\Psi\vert_{b}=\Psi_{b}$.\\
By Lemma \ref{lem:psib} we can apply Lefschetz Theorem on $(1,1)$ classes (\cite[Proposition 3.3.2]{huybrechts2006complex}) on $\Psi_{b}$ for each $S_{b}$ with $b\in B$, hence $\Psi_{b}$ comes from an algebraic class of the form
$\sum_{i,j} n_{i,j}Z_{b}^{i}\times Z^{j}_{b}$. 
We get that there exists a divisor 
\[
	Y_{b}=\cup_{n_{i,j}\neq 0}Z^{i}_{b}\subset S_{b}
\]
 and a cycle $\gamma_{b}\in A^{2}(S_{b}\times S_{b})_{\mathbb{Q}}$ such that $\mathrm{Supp} (\gamma_{b})\subseteq Y_{b}\times Y_{b}$ and 
\[
\left[\Psi_{b}\right]=\left[\gamma_{b}\right]\in H^{4}(S_{b}\times S_{b},\mathbb{Q}).
\]  
Now we apply Voisin's ``spreading of cycles'' \cite[Proposition 3.7]{voisin2013generalized}, to get that $\gamma_{b}$ exists relatively. 

\begin{prop}[Voisin's spreading of cycles]
Assume that for a very general point $b\in B$, there exists a divisor $Y_{b}\subset S_{b}$, and an algebraic cycle $\gamma_{b}\subset Y_{b}\times Y_{b}$ with $\mathbb{Q}$--coefficients,
such that
\[
	[\Psi_{b}] = [\gamma_{b} ] \hbox{ in } H^{2k} (S_{b} \times S_{b} , \mathbb{Q}).
\]
Then there exists a divisor $\mathcal{Y}\subset \mathcal{S}$ and a cycle $\gamma\in A^{2}(\mathcal{S}\times_{B}\mathcal{S})_{\mathbb{Q}}$ supported on $\mathcal{Y}\times_{B}\mathcal{Y}$ such that
\[
\left[\Psi_{b}\right]=\left[\gamma\vert_{b}\right]\in H^{4}(S_{b}\times S_{b},\mathbb{Q}).
\]
\end{prop}
Finally we can define the correspondence
\[
\Psi':=\Psi-\gamma=(2\Delta_{\mathcal{S}}-^{t}\Gamma_{f}\circ \Gamma_{f})\circ \pi^{\mathcal{S}}_{2}-\gamma \in A^{2}(\mathcal{S}\times_{B}\mathcal{S})_{\mathbb{Q}}.
\]
Then $\Psi'$ has the desired property of being homologically trivial when restricted to any fiber, i.e.  for any $b\in B$
\[
\left[\Psi'\vert_{b}\right]=\left[\Psi_{b}\right]-\left[\gamma\vert_{b}\right]=0\in H^{4}(S_{b}\times S_{b},\mathbb{Q}).
\]

Now we want to apply the Leray spectral sequence argument as in \cite[proof of Theorem 3.1]{laterveer2018algebraic}. We recall a result due to Voisin \cite[Lemma 3.11,3.12]{voisin2013generalized} which gives a decomposition of a homologically trivial cycle which vanishes on the fibers.
\begin{lem}
\label{lem:Leray}
Let $\left[\Psi'\right]\in H^{4}(\mathcal{S}\times_{B}\mathcal{S}, \mathbb{Q})$ be a fiberwise homologically trivial cohomology class, i.e. $\left[\Psi'\vert_{b}\right]=0$ for any $b\in B$. Then 
\[
\left[\Psi'\right]=\beta_{1}\vert_{\mathcal{S}\times_{B}\mathcal{S}}+\beta_{2}\vert_{\mathcal{S}\times_{B}\mathcal{S}},
\]
where $\beta_{1}\in H^{4}((\mathbb{P}^{6}/G)\times \mathcal{S}, \mathbb{Q})$ and  $\beta_{2}\in H^{4}( \mathcal{S}\times( \mathbb{P}^{6}/G), \mathbb{Q})$.\\
Moreover, since $\mathbb{P}^{6}/G$ has trivial Chow groups and $\left[\Psi'\right]$ is algebraic, we can choose $\beta_{1},\beta_{2}$ to be the classes of the restriction of algebraic cycles on $B\times (\mathbb{P}^{6}/G)\times (\mathbb{P}^{6}/G)$. 
\end{lem}

\begin{rem}
The proof of Lemma \ref{lem:Leray} relays on the analysis of the graded pieces of the Leray filtration given by the Leary spectral sequence of $(\pi,\pi)\colon \mathcal{S}\times_{B}\mathcal{S}\to B$. 
Note that in Voisin’s paper \cite{voisin2013generalized} the ambient space is smooth, whereas in our case it is a quotient variety $\mathbb{P}^6/G$. However, it is still possible to apply the result in our setting since we can move everything to $\mathbb{P}^6$ (which is smooth and so Voisin’s argument applies), and use that
$H^\ast( \mathbb{P}^{6}/G,\mathbb{Q})$ is just the $G$--invariant part of $H^\ast(\mathbb{P}^{6},\mathbb{Q})^G$. Since the argument in the proof of Lemma \ref{lem:Leray} is compatible with the action of $G$, it also apply to our case in which the singularities are only rational.
 Moreover, we notice that $\mathbb{P}^{6}/G$ has trivial Chow groups. 
\end{rem}

By means of Lemma \ref{lem:Leray}, we have 
\[
\left[\Psi'\right]=\beta_{1}\vert_{\mathcal{S}\times_{B}\mathcal{S}}+\beta_{2}\vert_{\mathcal{S}\times_{B}\mathcal{S}}=\left[\alpha_{1}\right]\vert_{\mathcal{S}\times_{B}\mathcal{S}}+\left[\alpha_{2}\right]\vert_{\mathcal{S}\times_{B}\mathcal{S}}
\]
with $\beta_{i}=\left[\alpha_{i}\right]\vert_{\mathcal{S}\times_{B}\mathcal{S}}$ and $\alpha_{i}\in A^{2}(B\times (\mathbb{P}^{6}/G)\times( \mathbb{P}^{6}/G))$.
We can define 
\[
\left[\Psi''\right]=\left[\Psi'\right]-\left(\left[\alpha_{1}\right]+\left[\alpha_{2}\right]\right)\vert_{\mathcal{S}\times_{B}\mathcal{S}}=0\in H^{4}(\mathcal{S}\times_{B}\mathcal{S}, \mathbb{Q}).
\]
We notice that $\left[\Psi''\right]$ is algebraic because it's the difference between algebraic cycles, so  $\Psi''\in A^{2}_{hom}(\mathcal{S}\times_{B}\mathcal{S})_{\mathbb{Q}}=0$, where the last equality holds by Proposition \ref{prop:key}.\\
 Then we have that 
 \begin{gather*}
 \Psi''=0 \hbox{ in } A^{2}_{hom}(\mathcal{S}\times_{B}\mathcal{S})_{\mathbb{Q}},\\
 \Psi'=(2\Delta_{\mathcal{S}}-^{t}\Gamma_{f}\circ \Gamma_{f})\circ \pi^{\mathcal{S}}_{2}-\gamma=(\alpha_{1}+\alpha_{2})\vert_{\mathcal{S}\times_{B}\mathcal{S}} \hbox{ in } A^{2}_{hom}(\mathcal{S}\times_{B}\mathcal{S})_{\mathbb{Q}}.
 \end{gather*}
 When we restrict to each fiber, and we look at the action on cycles, we get $\forall b\in B$:
 \begin{align*}
 	2\id_{*}=&(2\Delta_{S_{b}}\circ \pi_{2}^{b})_{*}\colon A_{\mathrm{hom}}^{2}(S_{b})_{\mathbb{Q}}\to 
	A_{\mathrm{hom}}^{2}(S_{b})_{\mathbb{Q}} \\
	=&(f_{b})^{*}(f_{b})_{*}(\pi_{2}^{b})_{*}+(\gamma_{b})_{*}+
	(\alpha_{1}+\alpha_{2})\vert_{S_{b}\times S_{b},\ast}(S_{b})_{\mathbb{Q}}\\
	=&(f_{b})^{*}(f_{b})_{*}+(\gamma_{b})_{*}+
	(\alpha_{1}+\alpha_{2})\vert_{S_{b}\times S_{b},\ast}(S_{b})_{\mathbb{Q}}
	,
 \end{align*}
 where last equality holds since $\pi_{2}^{b}$ acts as the identity on $A^2_{\mathrm{hom}}(S_{b})_{\mathbb{Q}}$.
 We recall that $\gamma_{b}$ is supported on a divisor,  
 hence it acts trivially on $A_{0}(S_{b})$ and $\alpha_{1}+\alpha_{2}\in A^{2}(B\times (\mathbb{P}^{6}/G)\times (\mathbb{P}^{6}/G))$. So on the right, the only term that acts on homological trivial $0$--cycles is $(f_{b})^{*}(f_{b})_{*}$. We get
 \[
 	(f_{b})^{*}(f_{b})_{*}=2\id_{*}\colon A_{\mathrm{hom}}^{2}(S_{b})\to A_{\mathrm{hom}}^{2}(S_{b}) \quad \text{where} 
	\begin{tikzcd}
				\sigma \curvearrowright S_{b}=V_{b}/G 
				\arrow[d, "2:1", "f_{b}"' ]\\
			M_{b}.
		\end{tikzcd}
\]
Since by projection formula we have $(f_{b})_{*}(f_{b})^{*} = \deg(f_{b})\id$, we conclude that $A^{2}_{\mathrm{hom}}(S)_{\mathbb{Q}}\cong A^{2}_{\mathrm{hom}}(M)_{\mathbb{Q}}.$

\end{proof}

\section{Motives}
\label{sec:mot}
Here we present the motivic version of Theorem \ref{teo:Tod} with some interesting corollaries. The central result is that a Todorov surface of type $(2,12)$ has the transcendental part of the motive isomorphic to the associated K3 surface's one (in the sense of \cite{MR2187153}). \\

First of all we briefly recall the definition of the Chow--Künneth decomposition, which always exists for a smooth projective surface (see \cite{MR1061525}, \cite[Proposition 2.1]{MR2187153}).
\begin{defn}
\label{defn:CKd}
Let $S$ be a smooth projective surface and let $h(S)\in \mathcal{M}_{rat}$ denote the Chow motive\footnote{For the definition of Chow motive of a smooth projective variety see for example \cite[Chapter 2]{murre2013lectures}.} of $S$. Then there exists a \emph{Chow--Künneth decomposition} of $h(s)$ in $\mathcal{M}_{rat}$
\[
	h(S)=\bigoplus_{i=0}^{4}h_{i}(S),
\]
where $h_{i}(S)=(S,\pi_{i},0)$, $\pi_{i}\in A^{2}(S\times S)$ are orthogonal projectors, i.e. $\pi_{i}\circ\pi_{i}=\pi_{i}$ and $\pi_{i}\circ\pi_{j}=0$ for $i\neq j$, and they are the Künneth components of the diagonal $\Delta_{S}$, i.e.
\begin{gather*}
	[\Delta_{S}]=\sum_{i=0}^{4}[\pi_{i}]\in H^{4}(S\times S,\mathbb{Q}),\\
	\cl^{2}(\pi_{i})\in H^{4-i}(S,\mathbb{Q})\otimes H^{i}(S,\mathbb{Q})\subset H^{4}(S\times S,\mathbb{Q}).
\end{gather*}
In particular, this decomposition is self--dual in the sense that $\pi_{i}=\pi^{t}_{4-i}$ (where $\pi^{t}_{4-i}$ denotes the transpose correspondence of $\pi_{4-i}$).
\end{defn}

In order to study the groups of $0$--cycles $A_{0}(S)$, Bloch's conjecture suggests that the interesting part of this decomposition is $h_{2}(S)=(S,\pi_{2},0)$ where
$\pi_{2}=\Delta_{S}-\pi_{0}-\pi_{1}-\pi_{3}-\pi_{4}$. To study this summand we use a further decomposition due to Kahn--Murre--Pedrini \cite[Proposition 2.3]{MR2187153}.

\begin{prop}[Kahn--Murre--Pedrini]
\label{prop:tr}
Let $S$ be a smooth projective surface with a Chow--Künneth decomposition as in Definition \ref{defn:CKd}. There there is a unique splitting in orthogonal projectors
\[
	\pi_{2}=\pi_{2}^{\mathrm{alg}}+\pi_{2}^{\mathrm{tr}} 
	\hbox{ in } A^{2}(S\times S)_{\mathbb{Q}}.
\]	
This gives an induced decomposition on the motive
\[
	h_{2}(S)\cong h_{2}^{\alg}\oplus t_{2}(S) \hbox{ in } \mathcal{M}_{rat},
\]
where $h_{2}^{\alg}(S)=(S,\pi_{2}^{\alg},0)$, $t_{2}(S,\pi_{2}^{\tr},0)$ and in cohomology we get
\[
	H^{*}(t_{2}(S),\mathbb{Q})=H^{2}_{\tr}(S), 
	\quad H^{*}(h^{\alg}_{2}(S),\mathbb{Q})=NS(S)_{\mathbb{Q}},
\]
where the transcendental cohomology $H^{2}_{\tr}(S)$ is defined as the orthogonal complement of the Néron--Severi group $NS(S)_{\mathbb{Q}}$ in $H^{2}(S,\mathbb{Q})$.\\ 
Moreover, we have that $A^{*}(t_{2}(S)=A^{2}_{AJ}(S))$.
\end{prop}

The component $t^{2}(S)$ is called the \emph{transcendental part of the motive} of $S$.

Next we recall a useful result on the Chow--Künneth decomposition (\cite[Theorem 3.10]{MR2187153}).
\begin{teo}[Kahn--Murre--Pedrini]
\label{teo:KMP}
Let $S$ and $S'$ be two smooth projective surfaces with a Chow--Künneth decomposition 
\[
	h(S)=\bigoplus_{i=0}^{4}h_{i}(S), \quad h(S')=\bigoplus_{i=0}^{4}h_{i}(S'),
\]
as in Definition \ref{defn:CKd}. Then
\[
	\mathcal{M}_{rat}(h_{i}(S),h_{j}(S'))=0 \text{ for all } j<i \text{ and } 0\leq i \leq 4,
\]
where $\mathcal{M}_{rat}(h_{i}(S),h_{j}(S'))=\pi_{i}(S)\circ A^{2}(S\times S')\circ \pi_{i}(S')$ are the morphisms in the category $\mathcal{M}_{rat}$.
\end{teo}

We are finally ready to prove that a Todorov surface of type $(2,12)$ has the transcendental part of the motive isomorphic to the associated K3 surface's one. The proof is directly inspired by Laterveer's work \cite{laterveer2018algebraic}.

\begin{teo}
\label{teo:mot}
Let $S$ be a Todorov surface of type $(2,12)$, and let $M$ be the K3 surface associated to $S$, i.e. the minimal resolution of $S/\sigma$. Then there is an isomorphism of Chow motives
\[
	t_{2}(S)\cong t_{2}(M) \hbox{ in } \mathcal{M}_{rat}.
\]
\end{teo}

\begin{proof}
By the description we did of the family of Todorov surface in Section \ref{subsec:desc}, we recall that fiberwise we have the following situation:
\[
		\begin{tikzcd}
			G \curvearrowright V_{b} \qquad\arrow[d, "4:1", "q_{b}"' ]\\
				\qquad \sigma \curvearrowright S_{b}=V_{b}/G 
				\arrow[d, "2:1", "f_{b}"' ]\\
		M_{b}.
		\end{tikzcd}
\]
By Theorem \ref{teo:desc} we have that $S\cong S_{b}$ for some $b\in B$ and $M\cong M_{b}$. Let us consider now the Chow--Künneth decomposition $\{\pi_{0}^{S_{b}},\pi_{2}^{S_{b}},\pi_{4}^{S_{b}}\} $for $S_{b}$ and $\{\pi_{0}^{M_{b}},\pi_{2}^{M_{b}},\pi_{4}^{M_{b}}\} $for $M_{b}$, as in Definition \ref{defn:CKd}. Then Proposition \ref{prop:tr} gives a further decomposition in the algebraic and the transcendental part of the second component:
\[
	\pi_{2}^{S_{b}}=\pi_{2}^{S_{b},\alg}+\pi_{2}^{S_{b},\tr} \hbox{ and } 
	\pi_{2}^{M_{b}}=\pi_{2}^{M_{b},\alg}+\pi_{2}^{M_{b},\tr}.
\]
Let us consider now the correspondence constructed in the proof of Theorem \ref{teo:cyc}:
\[
 	2\Delta_{S_{b}}\circ \pi_{2}^{S_{b}}
	=^{t}\Gamma_{b}\circ \Gamma_{b}\circ \pi_{2}^{S_{b}}+\gamma_{b}+(\alpha_{1}+\alpha_{2})\vert_{S_{b}\times S_{b}}\in A^{2}_{\mathrm{hom}}(S_{b}\times S_{b})_{\mathbb{Q}},
\]
where $\Gamma_{b}$ is the graph of $f_{b}$ and $^{t}\Gamma_{b}$ is its transpose.
We apply to this the composition with the correspondence $\pi_{2}^{S_{b},\tr}$ on both sides:
\begin{equation}
\label{eq:split}
\begin{split}
	2\pi_{2}^{S_{b},\tr}
	&=\pi_{2}^{S_{b},\tr}\circ 2\Delta_{S_{b}}\circ \pi_{2}^{S_{b}}
	\circ \pi_{2}^{S_{b},\tr}\\
	&=\pi_{2}^{S_{b},\tr}\circ \left(^{t}\Gamma_{b}\circ 
	\Gamma_{b}\circ \pi_{2}^{S_{b}}+
	\gamma_{b}+(\alpha_{1}+\alpha_{2})\vert_{S_{b}\times S_{b}}\right)
	\circ \pi_{2}^{S_{b},\tr}\\
	&=\left(\pi_{2}^{S_{b},\tr}\circ ^{t}\Gamma_{b}\circ \Gamma_{b}\circ \pi_{2}^{S_{b}}
	\circ \pi_{2}^{S_{b},\tr}\right)
	+\left(\pi_{2}^{S_{b},\tr}\circ \gamma_{b}\circ \pi_{2}^{S_{b},\tr}\right)\\
	&\quad\qquad \qquad\qquad\qquad+\left(\pi_{2}^{S_{b},\tr}\circ(\alpha_{1}+\alpha_{2})\vert_{S_{b}\times S_{b}}
	\circ\pi_{2}^{S_{b},\tr}\right).
\end{split}
\end{equation}

We recall that $\gamma_{b}$ is supported on $Y_{b}\times Y_{b}$ where $Y_{b}\subset S_{b}$ is a divisor, so $\gamma_{b}$ is not dominant over either factor of $S_{b} × S_{b}$ and hence acts as $0$ (see \cite[Theorem 4.3]{MR2187153}) and can be omitted from the calculations.
In particular, this shows that $\pi_{2}^{S_{b},\tr}\circ \gamma_{b} \circ \pi_{2}^{S_{b},\tr}=0$ in $\mathcal{M}_{rat}(t_{2}(S_{b}),t_{2}(S_{b}))$.

Next we recall that $\alpha_{i}\in A^{2}(B\times \mathbb{P}^{6}\times \mathbb{P}^{6})$. So we can write
\[
	(\alpha_{1}+\alpha_{2})\vert_{S_{b}\times S_{b}}=\sum_{i,j}D_{i}\times D_{j}=\left(\sum_{i,j}D_{i}\times D_{j}\right)\circ \pi_{2}^{S_{b},\alg},
\]
where $D_{i},D_{j}\subset S_{b}$ are divisors and the last equality holds since $\pi_{2}^{S_{b},\alg}$ is a projector on the Neron--Severi group $NS(S_{b})_{\mathbb{Q}}$. Being $\pi_{2}^{S_{b},\alg}$ and $\pi_{2}^{S_{b},\tr}$ orthogonal we conclude that
\begin{equation*}
	\left(\pi_{2}^{S_{b},\tr}\circ(\alpha_{1}+\alpha_{2})\vert_{S_{b}\times S_{b}}
	\circ\pi_{2}^{S_{b},\tr}\right)=
	\left(\pi_{2}^{S_{b},\tr}\circ\left(\sum_{i,j}D_{i}\times D_{j}\right)\circ \pi_{2}^{S_{b},\alg}\circ\pi_{2}^{S_{b},\tr}\right)=0.
\end{equation*}

So in \eqref{eq:split} the only summand that survives on the left is the first one, and we get
\begin{equation}
\label{eq:S}
		2\pi_{2}^{S_{b},\tr}=\pi_{2}^{S_{b},\tr}\circ ^{t}\Gamma_{b}\circ \Gamma_{b}\circ \pi_{2}^{S_{b}}
	\circ \pi_{2}^{S_{b},\tr}=\pi_{2}^{S_{b},\tr}\circ ^{t}\Gamma_{b}\circ \Gamma_{b}	\circ \pi_{2}^{S_{b},\tr},
\end{equation}
where last equality holds since $\pi_{2}^{S_{b}}=\pi_{2}^{S_{b},\alg}+\pi_{2}^{S_{b},\tr}$ and $\pi_{2}^{S_{b},\alg}, \pi_{2}^{S_{b},\tr}$ are orthogonal.
Next we claim that
\begin{equation}
\label{eq:cl}
	2\pi_{2}^{S_{b},\tr}=\pi_{2}^{S_{b},\tr}\circ ^{t}\Gamma_{b}\circ \pi_{2}^{M_{b},\tr} \circ \Gamma_{b}	\circ \pi_{2}^{S_{b},\tr} \hbox{ in } \mathcal{M}_{rat}(t_{2}(S_{b}),t_{2}(S_{b})).
\end{equation}
To prove the claim we recall that $\pi_{2}^{M_{b},\alg}$ and $\pi_{2}^{M_{b},\tr}$ are orthogonal and $\pi_{2}^{M_{b}}=\pi_{2}^{M_{b},\alg}+\pi_{2}^{M_{b},\tr}$, thus we get
\begin{align*}
	\pi_{2}^{S_{b},\tr}\circ ^{t}\Gamma_{b}\circ \pi_{2}^{M_{b},\tr} \circ \Gamma_{b}	
	\circ \pi_{2}^{S_{b},\tr}&=
	\pi_{2}^{S_{b},\tr}\circ ^{t}\Gamma_{b}\circ \pi_{2}^{M_{b}} \circ \Gamma_{b}	
	\circ \pi_{2}^{S_{b},\tr}\\
	&=
	\pi_{2}^{S_{b},\tr}\circ ^{t}\Gamma_{b}\circ 
	\left(\Delta_{M_{b}}-\pi_{0}^{M_{b}}-\pi_{4}^{M_{b}}\right) \circ \Gamma_{b}	
	\circ \pi_{2}^{S_{b},\tr}\\
	&=\pi_{2}^{S_{b},\tr}\circ ^{t}\Gamma_{b}\circ 
	\Delta_{M_{b}} \circ \Gamma_{b}	
	\circ \pi_{2}^{S_{b},\tr}=\pi_{2}^{S_{b},\tr}\circ ^{t}\Gamma_{b}\circ 
	\Gamma_{b}\circ \pi_{2}^{S_{b},\tr},
\end{align*} 
where the last equalities follow from Theorem \ref{teo:KMP}. Then we conclude the proof of the claim by means of \eqref{eq:S}.\\
Now we want to prove that, analogously, there is a rational equivalence of cycles
\begin{equation}
\label{eq:M}
	2\pi_{2}^{M_{b},\tr}=\pi_{2}^{M_{b},\tr}\circ \Gamma_{b}\circ \pi_{2}^{S_{b},\tr} \circ ^{t}\Gamma_{b}	\circ \pi_{2}^{M_{b},\tr} \hbox{ in } A^{2}(M_{b}\times M_{b})_{\mathbb{Q}}.
\end{equation}
This follows easily since
\[
	2\Delta_{M_{b}}=\Gamma_{b}\circ ^{t}\Gamma_{b} \hbox{ in } 
	A^{2}(M_{b}\times M_{b})_{\mathbb{Q}}.
\]
So applying twice on both sides $\pi_{2}^{M_{b},\tr}$ we get:
\begin{align*}
	2\pi_{2}^{M_{b},\tr}
	&=\pi_{2}^{M_{b},\tr}\circ \Gamma_{b}\circ ^{t}\Gamma_{b}
		 \circ \pi_{2}^{M_{b},\tr}
	= \pi_{2}^{M_{b},\tr}\circ \Gamma_{b}\circ \Delta_{S_{b}} \circ ^{t}\Gamma_{b} \circ \pi_{2}^{M_{b},\tr}\\
	&=\pi_{2}^{M_{b},\tr}\circ \Gamma_{b}\circ 
	\left(\Delta_{S_{b}}-\pi_{0}^{S_{b}}-\pi_{4}^{S_{b}}\right) \circ ^{t}\Gamma_{b} \circ \pi_{2}^{M_{b},\tr}\\
	&=\pi_{2}^{M_{b},\tr}\circ \Gamma_{b}\circ 
	\left(\pi_{2}^{S_{b},\alg}+\pi_{2}^{S_{b},\tr} \right)\circ ^{t}\Gamma_{b}\circ \pi_{2}^{M_{b},\tr}\\
	&=\pi_{2}^{M_{b},\tr}\circ \Gamma_{b}\circ \pi_{2}^{S_{b},\tr} \circ ^{t}\Gamma_{b}	\circ \pi_{2}^{M_{b},\tr}.
\end{align*}
By \eqref{eq:S} and \eqref{eq:M}, we conclude that $\Gamma_{b}\colon t_{2}(S_{b})\to t_{2}(M_{b})$ in $\mathcal{M}_{rat}$ is an isomorphism of motives, and its inverse is its transpose $^{t}\Gamma_{b}$. Since the transcendental part of the motive is a birational invariant, $S_{b}$ is birational to $S$, and $M_{b}$ is birational to $M$. We conclude that there is also an isomorphism of motives
\[
	t_{2}(S)\cong t_{2}(M) \text{ in } \mathcal{M}_{rat}.
\]
\end{proof}

We present some corollaries of this result. 

\begin{cor}
\label{cor:iso}
Let $S, S'$ be two isogenous Todorov surfaces of type $(2,12)$, then they have isomorphic Chow motives, i.e.
\[
	h(S)\cong h(S')\hbox{ in } \mathcal{M}_{rat}.
\]
\end{cor}
\begin{proof}
Being $S$ and $S'$ isogenous means that there exists a Hodge isometry $\varphi \colon H^{2}(S,\mathbb{Q})\overset{\sim}{\rightarrow}H^{2}(S',\mathbb{Q})$, i.e. $\varphi$ is a isomorphism of $\mathbb{Q}$--vector spaces which is compatible with the Hodge structure and the cup product on both sides\footnote{For a discussion on the meaning and different uses of the term ``isogenous'' see \cite{MR899482}.}. This implies that there is a Hodge isometry on the transcendental cohomology $H^{2}_{\tr}(S)\cong H^{2}_{\tr}(S')$ and on the algebraic one $H^{2}_{\alg}(S)\cong H^{2}_{\alg}(S')$. Let us denote by $\overline{M},\overline{M}'$ the singular K3 surfaces associated to $S$ and $S'$ respectively, and by $M,M'$ their resolutions of singularities. Then we have an isogeny given by the pullback $H^{2}_{\tr}(S)\cong H^{2}_{\tr}(\overline{M})$, since $S$ is a double cover of $\overline{M}$. Another isogeny is given by the pullback $H^{2}_{\tr}(M)\cong H^{2}_{\tr}(\overline{M})$, since transcendental cohomology is invariant when resolving quotient singularities. By Theorem \ref{teo:mot}, since $H^{*}(t_{2}(S),\mathbb{Q})=H^{2}_{\tr}(S)$ and $H^{*}(t_{2}(M),\mathbb{Q})=H^{2}_{\tr}(M)$, we have also an isomorphism $H^{2}_{\tr}(S)\cong H^{2}_{\tr}(M)$. In particular, this isomorphism is compatible with the Hodge structure, since it comes from a correspondence, and it is compatible with the cup product. Thus we get also a Hodge isometry $H^{2}_{\tr}(M)\cong H^{2}_{\tr}(M')$. 
By Huybrechts result on the motivic Šafarevič conjecture  \cite[Theorem 0.2]{2017arXiv170504063H}, we have that this Hodge isometry can be lifted to an isomorphism of Chow motives, i.e. $h(M)\cong h(M')$, and in particular we get an isomorphism on the transcendental part of the motives $t_{2}(M)\cong t_{2}(M')$. Then, by Theorem \ref{teo:mot}, we get an isomorphism of motives $t_{2}(S)\cong t_{2}(S')$ and we conclude that $h(S)\cong h(S')$ in $\mathcal{M}_{rat}$.
\end{proof}

\begin{cor}
Let $S$ be a Todorov surface of type $(2,12)$. Assume that $P$ is a K3 surface such that there is a Hodge isometry $H^{2}_{\tr}(S)\cong H^{2}_{\tr}(P)$. Then, there is an isomorphism of Chow motives
\[
	t_{2}(S)\cong t_{2}(P) \hbox{ in } \mathcal{M}_{rat}.
\]
\end{cor}
\begin{proof}
Let $M$ be the K3 surface associated to $S$, then by Theorem \ref{teo:mot} we have an isomorphism $H^{2}_{\tr}(S)\cong H^{2}_{\tr}(M)$. As we noticed in the proof of Corollary \ref{cor:iso}, this isomorphism is is compatible with Hodge structure and cup product and so there is also a Hodge isometry $H^{2}_{\tr}(M)\cong H^{2}_{\tr}(M')$.
Applying Huybrechts result \cite[Theorem 0.2]{2017arXiv170504063H} we can lift this isometry to an isomorphism of motives $t_{2}(M)\cong t_{2}(P)$ in $\mathcal{M}_{rat}$. By Theorem \ref{teo:mot} we conclude that $t_{2}(S)\cong t_{2}(M)\cong t_{2}(P)$ in $\mathcal{M}_{rat}$.
\end{proof}

\begin{cor}
Let $S$ be a Todorov surface of type $(2,12)$ with very high Picard number, i.e. $\rho(S)\geq h^{1,1}(S)-1$. Then $S$ has finite dimensional motive (in the sense of Kimura and O'Sullivan \cite{MR2167204}, \cite{MR2107443}).
\end{cor}
\begin{proof}
By \cite[Lemma 7.6.6]{MR2187153} the motives $h_{0}(S), h_{4}(S), h_{2}^{\alg}(S)$ are finite--dimensional, hence all the summands of the Chow motive $h(S)$ are finite--dimensional except perhaps $t_{2}(S)$. Since a direct sum of finite--dimensional motives is finite--dimensional,
it is enough to prove that $t_{2}(S)$ is finite--dimensional. Let $M$ be the K3 surface associated to $S$. By Theorem \ref{teo:mot} we have $t_{2}(S)\cong t_{2}(M)$, and so it suffices to show that $t_{2}(M)$ is finite--dimensional. \\
We recall that the Picard number of $S$, $\rho(S)$, is the rank of the Neron--Severi group $NS(S)_{\mathbb{Q}}$, and $\dim H^{2}_{\tr}(S)=b_{2}(S)−\rho(S)=2-\rho(S)\leq 3-h^{1,1}(S)\leq 3$, since by hypothesis $\rho(S)\geq h^{1,1}(S)-1$.
By the isomorphism $H^{2}_{\tr}(S)\cong H^{2}_{\tr}(M)$ we get that $\rho(M)\geq H^{2}_{\tr}(M)-3=b_{2}(M)-3=19$. Since $M$ has a large Picard number, it has finite dimensional motive \cite[Theorem 2]{MR2898747}.
\end{proof}

\vskip0.8cm

\begin{nonumberingt} Thanks to Robert Laterveer for inspiring this work, helping me through it and for always pushing me. Thanks to Roberto Pignatelli for his great help and patience and thanks to Claudio Fontanari for his support. Thanks to Davide Frapporti for his last minute help. I'm deeply grateful to Milo, Elvio and Federico for always providing enthusiasm and inspiration.
\end{nonumberingt}

%

\bibliographystyle{amsplain}
\bibliography{rTodorov}

\end{document}